\documentclass{amsart}

\usepackage{epsfig,graphicx,amsfonts,amssymb,amsmath,amscd,amsthm,color,xy,pgf,tikz}
\usetikzlibrary{arrows}

\makeindex

\numberwithin{equation}{section}

\newtheorem{thm}{Theorem}[section]
\newtheorem{lem}[thm]{Lemma}
\newtheorem{prop}[thm]{Proposition}
\newtheorem{cor}[thm]{Corollary}

\theoremstyle{definition}
\newtheorem{defin}[thm]{Definition}

\newtheorem{exmp}[thm]{Example}

\theoremstyle{remark}
\newtheorem{rem}{Remark}

\newtheorem*{notation}{Notation}

\newcommand{\Ker}{{\rm Ker}}
\newcommand{\rank}{{\rm rank}}

\newcommand{\val}{{\rm val}}

\newcommand{\Span}{{\rm Span}}
\newcommand{\Spec}{{\rm Spec}}

\newcommand{\CCC}{{\mathbb C}}

\newcommand{\SSS}{{\mathbb S}}

\newcommand{\RR}{{\mathbb R}}
\newcommand{\ZZ}{{\mathbb Z}}
\newcommand{\PP}{{\mathbb P}}

\newcommand{\GG}{{\mathbb G}}

\newcommand{\QQ}{{\mathbb Q}}

\newcommand{\AAA}{{\mathbb A}}

\newcommand{\CW}{{\mathcal W}}
\newcommand{\CO}{{\mathcal O}}

\newcommand{\CM}{{\mathcal M}}

\newcommand{\CE}{{\mathcal E}}

\newcommand{\ft}{{\mathfrak t}}
\newcommand{\fh}{{\mathfrak h}}
\newcommand{\fe}{{\mathfrak e}}
\newcommand{\fm}{{\mathfrak m}}
\newcommand{\fr}{{\mathfrak r}}

\DeclareMathAlphabet{\mathpzc}{OT1}{pzc}{m}{it}

\newcommand{\ord}{{\rm ord}}

\def\:{\colon}
\def\val{\nu}
\def\oF{\overline{F}}

\def\blambda{\boldsymbol\lambda}
\def\balpha{\boldsymbol\alpha}
\def\bbeta{\boldsymbol\beta}

\def\bchi{\boldsymbol\chi}
\def\bO{\boldsymbol O}
\def\b1{\boldsymbol 1}
\def\bold0{\boldsymbol 0}
\def\bzeta{\boldsymbol\zeta}
\def\be{\boldsymbol e}
\def\bq{\boldsymbol q}
\def\by{\boldsymbol y}
\def\bphi{\boldsymbol \phi}

\def\bvarphi{\boldsymbol \varphi}
\def\bpsi{\boldsymbol \psi}
\def\Tr{\mathpzc{Tr}}
\def\tr{\mathpzc{tr}}

\begin{document}

\title{Enumeration of rational curves with cross-ratio constraints}
\author{Ilya Tyomkin}

\thanks{Partially supported by German-Israeli Foundation under grant agreement 1174-197.6/2011.}

\address{Department of Mathematics, Ben-Gurion University of the Negev, P.O.Box 653, Be'er Sheva, 84105, Israel}
\email{tyomkin@math.bgu.ac.il}

\begin{abstract}
In this paper we prove the algebraic-tropical correspondence for stable maps of rational curves with marked points to toric varieties such that the marked points are mapped to given orbits in the big torus and in the boundary divisor, the map has prescribed tangency to the boundary divisor, and certain quadruples of marked points have prescribed cross-ratios. In particular, our results generalize the results of Nishinou-Siebert \cite{NS06}. The proof is very short, involves only the standard theory of schemes, and works in arbitrary characteristic (including the mixed characteristic case).
\end{abstract}
\keywords{Enumeration of rational curves, algebraic and tropical geometry.}
\maketitle

\section{Introduction}

\subsection{Background}
Enumeration of curves in algebraic varieties is a classical problem that has a long history going back to Ancient Greeks. Many tools have been developed to approach enumerative problems including Schubert calculus, intersection theory, degeneration techniques, quantum cohomology etc.

In 1989, Ran \cite{Ran89} proposed a recursive procedure based on degeneration techniques for enumeration of nodal curves of given degree and genus in the plane satisfying point constraints. Several years later, there was a major break-through in the problem, when Kontsevich introduced the moduli spaces of stable maps \cite{K95,KM94}, and used them to get recursive formulae for the number of rational plane curves of degree $d$ passing through $3d-1$ points in general position. Kontsevich interpreted the enumerative invariants as Gromov-Witten invariants, and deduced information about them from properties of the quantum cohomology.

Stated in more classical terms, Kontsevich's recursion can be obtained as follows: one considers the curves with $3d$ marked points such that the first two are mapped to two given lines, the remaining $3d-2$ to points in general position, and the cross-ratio with respect to the first four points is a parameter. Specialization of the cross ratio to the values $0$ and $1$ then gives the desired recursion.

Using various degeneration techniques recursive formulae for different target spaces and higher genera have been obtained in a series of works of Pandharipande \cite{Pan97}, Caporaso-Harris \cite{CH98-1,CH98-3}, Vakil \cite{Vak00-2}, and others.

In the beginning of the century, Mikhalkin \cite{Mik05} proposed a new approach to the problem. He introduced piece-wise linear objects, called tropical curves, and proved that there is a natural correspondence between algebraic and tropical curves, that in good cases allows one to obtain closed formulae for the number of algebraic curves in terms of the tropical counterpart. The tropical approach turned out to be a powerful tool also in mirror symmetry and in real algebraic geometry, where it was a break-through, and in particular, led to the calculation of Welschinger invariants in many interesting cases, see e.g., \cite{IKS13,Mik05,Shu06}.

Tropical geometry is now a rapidly developing field, that includes the study of the combinatorial piece-wise linear side of tropical varieties as well as the link between tropical, algebraic, and Berkovich analytic geometry. On one side, many tropical analogs of classical problems have been studied in the recent years. In particular, Gathmann and Markwig proved tropical analogs of formulae of Kontsevich and Caporaso-Harris \cite{GM07,GM08}. On the other side, several new proofs of known results such as Brill-Noether Theorem \cite{CDPR12} have been found, as well as tropical proofs of algebra-geometric statements that used to be beyond the reach of the classical methods, e.g., Zariski's theorem in positive characteristic \cite{Tyo13}.

Since 2005, few algebraic proofs of various versions of Mikhalkin's correspondence have been obtained by Nishinou-Siebert \cite{NS06}, Shustin \cite{Shu05}, the author \cite{Tyo12}, Ranganathan \cite{Ran15} and others. However, the proofs are relatively complicated and involve techniques such as deformation theory, log-geometry, stacks, rigid analytic spaces etc.; and all but \cite{Tyo12} assume the ground field to be of characteristic zero.

\subsection{The goals of the paper, the results, and the approach}
In the current paper we study the algebraic-tropical correspondence for stable maps of rational curves with marked points to toric varieties such that the marked points are mapped to given orbits of given subtori in the big torus and in the boundary divisor, the map has prescribed tangency to the boundary divisor, and certain quadruples of marked points have prescribed cross-ratios. In particular, our results extend the results of Nishinou-Siebert \cite{NS06}.

%First, in Section~\ref{sec:prob}, we set the natation, remind Mikhalkin's tropical analog of the cross-ratio invariant, and describe the objectives of the paper. Then, in Section~\ref{sec:trop},
We begin by proving that the canonical tropicalization procedure of \cite{Tyo12} associates to a stable map satisfying the constraints a rational parameterized tropical curve satisfying the tropicalization of the constraints.

Our first main result is the {\em Realization} theorem (Theorem~\ref{thm:realization}). Under certain regularity assumptions we prove that any parameterized tropical curve satisfying the tropicalization of the algebraic constraints belongs to the image of the tropicalization map. We shall emphasize, that we do not assume the tropical curve to be three-valent and our proof works over {\em any} complete discretely valued field, including the case of arbitrary small or mixed characteristic.

Our second main result is the {\em Correspondence} theorem (Theorem~\ref{thm:correspondence}). Assuming that the tropicalization of the constraints is tropically general, the characteristic of the residue field is big enough, and the problem is enumerative, we prove that the number of algebraic curves satisfying given constraints is equal to the number of tropical curves satisfying the tropicalization of the constraints and counted with explicit multiplicities. We also explain (Remark~\ref{rem:smallchar}) that if the characteristic of the residue field is arbitrary then the Correspondence theorem is still valid, but the algebraic curves should be counted with multiplicities too.

The proofs are surprisingly short, elementary, and involve no deformation theory, log-geometry, stacks, or rigid analytic spaces. Similarly to \cite{Ran15}, we do not use the degeneration of the target, but unlike the other proofs we use only the standard scheme theory, and if the characteristic of the residue field is big enough we even propose a reformulation in terms of elementary commutative algebra (see Subsection~\ref{subsec:algproof}).

Roughly speaking the proof of the Realization theorem consists of the following steps: First, we introduce convenient coordinates and describe the moduli space of the constrained stable maps that tropicalize to a given parameterized tropical curve as a fiber of an explicitly constructed map $\Theta$. This way the set of algebraic curves we are interested in becomes the set of integral points in a given fiber of $\Theta$. On the tropical side, we construct a linear map $\theta$ that controls the deformation theory of the parameterized tropical curve. Finally, we show that the map $\Theta$ is a deformation of the map of algebraic tori associated to $\theta$, which allows us to control the fibers of $\Theta$, and to deduce the result. The Correspondence theorem then follows from the Realization theorem, using an easy combinatorial lemma (Lemma~\ref{lem:trcurgenconstr}).

We assume that the reader is familiar with the theory of schemes, basic toric geometry, commutative algebra, and knows the definition of parameterized tropical curves. However, we remind the necessary notions from tropical geometry, in particular the canonical tropicalization procedure of \cite{Tyo12}, in Section~\ref{sec:trop}. We made a lot of effort to introduce intuitive and self-explaining notation, and summarized most of it in Subsection~\ref{sec:convnotset}. We shall especially emphasize \S\ref{subsubsec:rootedtree} which is used intensively in the definition of the coordinates we are working with. For the convenience of the reader we include an index of notation as was suggested by the anonymous referee.

To conclude, let us mention that the main difficulty with extending the approach of the current paper to the case of curves of higher genus is the lack of nice global coordinates similar to those introduced in the current paper. Nevertheless, working with local (Berkovich) analytic coordinates, it should be possible to prove various Realization theorems for higher genera. In particular, we plan to discuss the extension of our approach to the case of genus one curves with moduli constraints in a forthcoming paper.

\subsection*{Acknowledgements} This research was initiated while the author was visiting the Centre Interfacultaire Bernoulli in 2014 in the framework of the special program on {\em Tropical geometry in its complex and symplectic aspects}. I am very grateful to the organizers Grigory Mikhalkin and Ilia Itenberg for inviting me, and to CIB for its hospitality and fantastic research atmosphere. I would also like to thank Eugenii Shustin, Hannah Markwig, and Michael Temkin for helpful discussions.

\tableofcontents

\section{Preliminaries}\label{sec:prob}

\subsection{Conventions and Notation}\label{sec:convnotset}
\subsubsection{Multi-index} Finite collections of objects with given index set are denoted with bold letters, e.g., $\bO=(O_1,\dotsc,O_s)$, $\b1=(1,\dotsc,1)$, and $\balpha=(\alpha_\gamma)$, $\gamma\in E$.\index{$\balpha,\blambda,\bO,\bq,$ etc.}
\subsubsection{Valuations} We fix a complete discretely valued field $F$\index{$F$} having algebraically closed residue field $k$\index{$k$} and group of values $\ZZ$, and its algebraic closure $\oF$\index{$\oF$}. The valuation on $\oF$ extending the valuation on $F$ is denoted by $\val\: \oF\to \QQ\cup\{\infty\}$\index{$\val$}. Such valuation exists and unique by \cite[Corollary~2, p.425]{Bou}.
For a finite intermediate extension $F\subseteq K\subseteq \oF$, we denote by $K^o$\index{$K^{o}$} and $K^{oo}$\index{$K^{oo}$} its ring of integers and maximal ideal respectively, and by $\pi$\index{$\pi$} a uniformizer of $K^o$. The ramification index $\val(\pi)^{-1}$ is denoted by $\fe_K$\index{$\fe_K$}. To simplify the notation, the ring of integers of $\oF$ is denoted by $R$\index{$R$}.

If $X$ is a Noetherian integral scheme then any reduced codimension-one subscheme $Y\subset X$ whose generic point is regular in $X$ defines a discrete valuation on the field of rational functions $K(X)$ - {\em the order of vanishing along $Y$}. %\cite[Theorem~11.5 and its Corollary]{MatCRT}
We denote the latter valuation by $\ord_Y$.\index{$\ord_Y$}

\subsubsection{Toric geometry}
In this paper $M$ and $N$\index{$M$}\index{$N$} denote a pair of dual lattices of finite rank, and $T_N:=\Spec (\ZZ[M])$\index{$T_N$} the corresponding torus. For a ring $A$ we identify the $A$-points $\chi\in T_N(A)$ and the homomorphisms $\chi\:M\to A^\times$. We denote the base change of $T_N\to\Spec(\ZZ)$ to $\Spec(A)$ by $T_{N,A}:=\Spec(A[M])$\index{$T_{N,A}$}. For a sublattice $L\subseteq N$ we denote the annihilator of $L$ in $M$ by $L^0$\index{$L^0$}. For an abelian group $G$, we denote $N_G:=N\otimes_\ZZ G$\index{$N_G$} and $M_G:=M\otimes_\ZZ G$\index{$M_G$}, e.g., $(\ZZ^n)_\RR=\RR^n$.% is the $\RR$-vector space corresponding to $N$.

\subsubsection{Graphs}
In this paper all graphs are finite. The set of vertices of a graph $\Gamma$ is denoted by $V(\Gamma)$\index{$V(\Gamma)$}, and the set of edges by $E(\Gamma)$\index{$E(\Gamma)$}. The {\em valency} (or {\em degree}) of a vertex $w$ is denoted by $\deg(w)$\index{$\deg(w)$}. In metric trees, the geodesic path connecting vertices $w$ and $w'$ is denoted by $[w,w']$\index{$[w,w']$}, and the length of an edge $\gamma$ by $|\gamma|$\index{\textbar$\gamma$\textbar}.

\subsubsection{Tropical curves}

We follow the conventions of \cite{Tyo12}. In particular, the infinite vertices of tropical curves are totally ordered. We use notation $V^f(\Gamma)$\index{$V^f(\Gamma)$} (resp. $V^\infty(\Gamma)$\index{$V^\infty(\Gamma)$}) for the set of finite (resp. infinite) vertices, and $E^b(\Gamma)$\index{$E^b(\Gamma)$} (resp. $E^\infty(\Gamma)$\index{$E^\infty(\Gamma)$}) for the set of bounded (resp. unbounded) edges of $\Gamma$.

In this paper, the infinite vertices are {\em always} denoted by $u_1,\dotsc, u_r$\index{$u_i$}, the corresponding unbounded edges (or ends) by $e_1,\dotsc,e_r$\index{$e_i$}, and the finite vertices attached to the unbounded edges by $v_1,\dotsc, v_r$\index{$v_i$}. Notice that while $u_1,\dotsc, u_r$ are distinct, $v_1,\dotsc, v_r$ need not (and usually will not) be distinct, see Figure~\ref{fig:order}. To avoid confusion with the set of ends denoted by $\be$, we use letter $\gamma$ when referring to edges in general. Similarly, we use letter $w$ when referring to vertices in general.

We work exclusively with {\em $N_\QQ$-parameterized $\QQ$-tropical curves}, i.e., parameterized tropical curves whose bounded edges have rational lengths, and whose finite vertices are mapped to rational points of $N_\RR$. Since $N_\RR$ is the open tropical torus, we remove the infinite vertices from the tropical curves mapped to it without warning. We say that a $N_\QQ$-parameterized $\QQ$-tropical curve $h\:(\Gamma;\be)\to N_\RR$ is {\em defined over a field $K$} if $|\gamma|\in \val(K^\times)$ for all $\gamma\in E^b(\Gamma)$ and $\fe_Kh(w)\in N$ for all $w\in V^f(\Gamma)$.

\begin{figure}
\definecolor{qqqqff}{rgb}{0.,0.,1.}
\definecolor{qqzzqq}{rgb}{0.,0.6,0.}
\definecolor{ffqqqq}{rgb}{1.,0.,0.}
\begin{tikzpicture}[line cap=round,line join=round,>=triangle 45,x=1.0cm,y=1.0cm]
\clip(-1.32,-2.94) rectangle (3.24,3.38);
\draw (1.,-2.7)-- (1.,-1.);
\draw (1.,-1.)-- (-1.1,0.68);
\draw (1.,-1.)-- (2.86,-0.02);
\draw (1.,-1.)-- (1.6,0.74);
\draw (1.6,0.74)-- (2.9,2.32);
\draw (1.6,0.74)-- (0.28,1.8);
\draw (0.28,1.8)-- (1.64,3.02);
\draw (0.28,1.8)-- (-0.82,3.);
\draw (-1.45,1.16) node[anchor=north west] {$u_1$};
\draw (-1.3,3.46) node[anchor=north west] {$u_2$};
\draw (1.5,3.46) node[anchor=north west] {$u_3$};
\draw (2.7,2.8) node[anchor=north west] {$u_4$};
\draw (2.7,0.45) node[anchor=north west] {$u_5$};
\draw (1,-2.5) node[anchor=north west] {$u_6$};
\draw (1,-0.9) node[anchor=north west] {$v_1=v_5=v_6$};
\draw (1.6,0.9) node[anchor=north west] {$v_4$};
\draw (-1.1,2) node[anchor=north west] {$v_2=v_3$};
\draw (0.8,0.2) node[anchor=north west] {$\gamma_1$};
\draw (0.9,1.6) node[anchor=north west] {$\gamma_2$};
\draw (-0.45,2.98) node[anchor=north west] {$e_2$};
\draw (0.64,3.) node[anchor=north west] {$e_3$};
\draw (1.96,2.22) node[anchor=north west] {$e_4$};
\draw (1.9,0.1) node[anchor=north west] {$e_5$};
\draw (-0.6,0.56) node[anchor=north west] {$e_1$};
\draw (0.45,-1.9) node[anchor=north west] {$e_6$};
\begin{scriptsize}
\draw [fill=ffqqqq] (1.,-2.7) circle (1.5pt);
\draw [fill=qqqqff] (1.,-1.) circle (1.5pt);
\draw [fill=qqzzqq] (-1.1,0.68) circle (1.5pt);
\draw [fill=qqzzqq] (2.86,-0.02) circle (1.5pt);
\draw [fill=qqqqff] (1.6,0.74) circle (1.5pt);
\draw [fill=qqzzqq] (2.9,2.32) circle (1.5pt);
\draw [fill=qqqqff] (0.28,1.8) circle (1.5pt);
\draw [fill=qqzzqq] (1.64,3.02) circle (1.5pt);
\draw [fill=qqzzqq] (-0.82,3.) circle (1.5pt);
\end{scriptsize}
\end{tikzpicture}
\caption{A rational tropical curve and its rooted tree structure.}
\label{fig:order}
\end{figure}
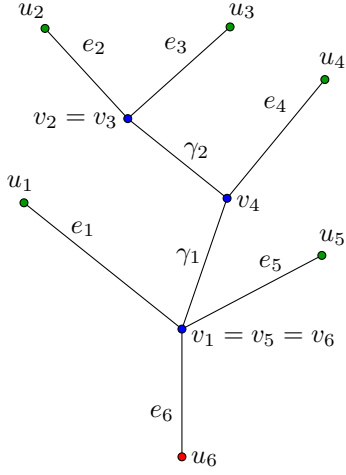

\subsubsection{The rooted tree structure and the induced combinatorics}\label{subsubsec:rootedtree}
Let $\Gamma$ be a rational tropical curve, and $u_1,\dotsc,u_r$ its infinite vertices. We declare $u_r$ to be the root, and orient the edges away from the root. The corresponding partial order on the tree $\Gamma$ for which the root is minimal, is denoted by $\preceq$\index{$\preceq$}. The tail and the head functions are denoted by $\ft\:E(\Gamma)\to V(\Gamma)$ and $\fh\:E(\Gamma)\to V(\Gamma)$ respectively.

For $w\in V^f(\Gamma)$, we set $E_w^+:=\{\gamma\,|\, \ft(\gamma)=w\}$\index{$E_w^+$} and $I_w^\infty:=\{i\,|\,u_i\succ w\}$\index{$I_w^\infty$}. Then $\{E_w^+\}_{w\in V^f(\Gamma)}$ is a partition of $E(\Gamma)\setminus\{e_r\}$, and $|E_w^+|=\deg(w)-1$. We also define a function $\iota\:E(\Gamma)\setminus\{e_r\}\to \{1,\dotsc, r-1\}$\index{$\iota$} by setting $\iota(\gamma)$ to be the minimal index $i$ for which $u_i\succeq \fh(\gamma)$. Then $\iota\:E_w^+\to \{1,\dotsc,r-1\}$ is injective for any $w\in V^f(\Gamma)$, hence induces a total order on each $E_w^+$.

We call an edge $\gamma\in E(\Gamma)$ {\em essential} if it belongs to some $E_w^+$ and is neither maximal nor minimal in it, and {\em inessential} otherwise; e.g., if $\deg(w)=3$ for all $w\in V^f(\Gamma)$ then all edges are inessential. The set of essential edges is denoted by $E^{es}(\Gamma)$\index{$E^{es}(\Gamma)$}. Finally, we set $I_w:=\iota(E_w^+)\subseteq I_w^\infty$\index{$I_w$}. Then $I_{\ft(\gamma)}\cap I_{\fh(\gamma)}=\{\iota(\gamma)\}$ and $\iota(\gamma)\in I_{\fh(\gamma)}$ is minimal for any $\gamma\in E^b(\Gamma)$.

\begin{exmp}\label{ex:rtstrandcomb}
To illustrate the definitions consider the curve on Figure~\ref{fig:order}. Then,
\begin{itemize}
\item the root is $u_6$ and the edges are oriented upwards;
\item $v_2\succeq v_4$, $u_4\succeq v_4$, but $v_2$ and $u_4$ are incomparable;
\item $\ft(\gamma_1)=v_1$ and $\fh(\gamma_1)=v_4$;
\item $E_{v_1}^+=\{e_1,e_5,\gamma_1\}$, $E_{v_2}^+=\{e_2,e_3\}$, and $E_{v_4}^+=\{e_4,\gamma_2\}$.%, and hence we see that the collection of subsets $E_{v_1}^+, E_{v_2}^+, E_{v_4}^+\subseteq E(\Gamma)\setminus\{e_6\}$ is a partition;
\item $I_{v_1}^\infty=\{1,2,3,4,5\}$ and $I_{v_4}^\infty=\{2,3,4\}$;
\item $\iota(\gamma_1)=\iota(\gamma_2)=2$ and $\iota(e_i)=i$ for all $i\le 5$;
\item the edge $\gamma_1$ is essential, and $\gamma_2$ is not;
\item finally, $I_{v_1}=\{1,2,5\}$, $I_{v_4}=\{2,4\}$.%, and indeed $I_{\ft(\gamma)}\cap I_{\fh(\gamma)}=\{2\}=\{\iota(\gamma)\}$.
\end{itemize}
\end{exmp}

\subsection{The cross-ratio}
\subsubsection{The algebraic cross-ratio}
Recall that the classical {\em cross-ratio} (or {\em double ratio}) of four distinct points $p_1,\dotsc, p_4\in\PP^1$ is defined by the formula
$$\index{$\lambda(\PP^1; p_1,\dotsc,p_4)$}
\lambda(\PP^1; p_1,\dotsc,p_4)=\frac{(p_3-p_1)(p_4-p_2)}{(p_4-p_1)(p_3-p_2)}.
$$
In particular, if $(p_2,p_3,p_4)=(1,0,\infty)$ then $p_1=\lambda(\PP^1; p_1,\dotsc,p_4)$. It is well-known that $\lambda$ is a coordinate on the moduli space of smooth connected rational curves with four marked points $\lambda\:\CM_{0,4}\xrightarrow{\sim}\PP^1\setminus \{0,1,\infty\}$.

\subsubsection{The tropical cross-ratio}
Let $(\Gamma; \be)$ be a rational tropical curve. In \cite{Mik07}, Mikhalkin defined the {\em tropical double ratio} of the two pairs $\{e_{i1}, e_{i2}\}$ and $\{e_{i3}, e_{i4}\}$ (or the {\em tropical cross ratio} of $\Gamma$ with respect to $e_{i1},\dotsc,e_{i4}$) to be the signed length of the intersection of the oriented geodesic paths joining $e_{i1}$ to $e_{i2}$ and $e_{i3}$ to $e_{i4}$, where the sign is positive if and only if the orientations are compatible. Plainly, the tropical cross-ratio is stable under tropical modifications/contractions.

We say that $\gamma\in E^b(\Gamma)$ {\em separates} $e_t, e_j$ from $e_d, e_l$ if and only if $e_t, e_j$ belong to one of the two connected components of $\Gamma\setminus\{\gamma\}$, and $e_d, e_l$ to another. Then the tropical cross-ratio of $\Gamma$ with respect to $e_{i1},\dotsc,e_{i4}$ is given by
$$\index{$\lambda^{\tr}(\Gamma; e_1,\dotsc,e_4)$}
\lambda^{\tr}(\Gamma; e_{i1},\dotsc,e_{i4}):=\sum_{\gamma\in E^b(\Gamma)}\epsilon(\gamma,i)|\gamma|;
$$
where
\begin{equation}\label{eq:epsilonei}\index{$\epsilon(\gamma,i)$}
\epsilon(\gamma,i)=\left\{
                  \begin{array}{ll}
                    1, & \gamma\; \hbox{separates the ends}\; e_{i1}, e_{i3}\; \hbox{from}\; e_{i2}, e_{i4}, \\
                    -1, & \gamma\; \hbox{separates the ends}\; e_{i1}, e_{i4}\; \hbox{from}\; e_{i2}, e_{i3}, \\
                    0 & \hbox{otherwise;}
                  \end{array}
                \right.
\end{equation}

\begin{figure}
\definecolor{qqqqff}{rgb}{0.,0.,1.}
\begin{tikzpicture}[line cap=round,line join=round,>=triangle 45,x=1.0cm,y=1.0cm]
\clip(-7.792847284651015,0.9139003520163966) rectangle (1.3560490131484155,4.42900261380248);
\draw (-7.38,3.74)-- (-6.4,2.58);
\draw (-3.42,3.06)-- (-0.44,2.58);
\draw (-6.4,2.58)-- (-7.34,1.66);
\draw (-7.5,2.4) node[anchor=north west] {$e_1$};
\draw (-7.2,3.9) node[anchor=north west] {$e_5$};
\draw (0.5,2.3) node[anchor=north west] {$e_3$};
\draw (-3.2,4.1) node[anchor=north west] {$e_4$};
\draw (-0.1,3.9) node[anchor=north west] {$e_2$};
\draw (-6.4,2.58)-- (-3.42,3.06);
\draw (-5.3,3.4) node[anchor=north west] {$\gamma_1$};
\draw (-3.42,3.06)-- (-3.44,4.04);
\draw (-1.9,3.4) node[anchor=north west] {$\gamma_2$};
\draw (-0.44,2.58)-- (0.72,3.64);
\draw (-0.44,2.58)-- (0.74,1.64);
\begin{scriptsize}
%\draw [fill=qqqqff] (-7.38,3.74) circle (2.5pt);
\draw [fill=qqqqff] (-6.4,2.58) circle (1.5pt);
\draw [fill=qqqqff] (-3.42,3.06) circle (1.5pt);
\draw [fill=qqqqff] (-0.44,2.58) circle (1.5pt);
%\draw [fill=qqqqff] (-7.34,1.66) circle (2.5pt);
%\draw [fill=qqqqff] (-3.44,4.04) circle (2.5pt);
%\draw [fill=qqqqff] (0.72,3.64) circle (2.5pt);
%\draw [fill=qqqqff] (0.74,1.64) circle (2.5pt);
\end{scriptsize}
\end{tikzpicture}
\caption{A stable rational tropical curve with 5 marked ends.}
\label{fig:picTrCR}
\end{figure}
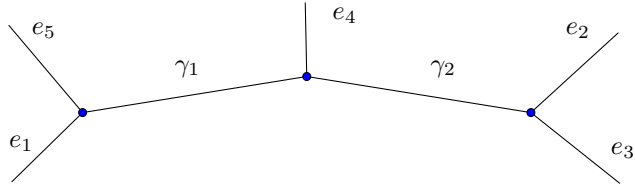

\begin{exmp}
Consider the rational tropical curve on Figure \ref{fig:picTrCR}. The common part of the geodesic paths joining $e_1$ to $e_2$ and $e_3$ to $e_4$ has length $|\gamma_2|$, but the orientations are opposite. Hence $\lambda^{\tr}(\Gamma;e_1,e_2,e_3,e_4)=-|\gamma_2|$. Alternatively, $\gamma_2$ separates $e_1,e_4$ from $e_2,e_3$, and $\gamma_1$ separates $e_1$ from $e_2,e_3,e_4$, which also leads to the conclusion that $\lambda^{\tr}(\Gamma;e_1,e_2,e_3,e_4)=0\cdot|\gamma_1|+(-1)\cdot|\gamma_2|=-|\gamma_2|.$ In a similar way one checks that $\lambda^{\tr}(\Gamma;e_1,e_4,e_2,e_3)=0$ and $\lambda^{\tr}(\Gamma;e_1,e_2,e_5,e_3)=|\gamma_1|+|\gamma_2|$.
\end{exmp}

\subsection{The objectives of the paper}
\subsubsection{The setting}
We fix once and for all
\begin{itemize}
  \item a pair of dual lattices $N$ and $M$;
  \item $n_1,\dotsc,n_r\in N$\index{$n_i$} such that\footnote{We do not assume that $n_i$ are primitive or non-zero.} $\sum_{i=1}^r n_i=0$;
  \item sublattices $n_i\in L_i\subseteq N$,\index{$L_i$} $1\le i\le r$, such that $N/L_i$ are torsion-free;
  \item $\oF$-points $\zeta_i\in (T_N/T_{L_i})(\oF)$\index{$\zeta_i$} for all $1\le i\le r$;
  \item $\oF$-point $\blambda=(\lambda_1,\dotsc, \lambda_s)\in T_{\ZZ}^s(\oF)$\index{$\lambda_i$};
  \item $s\times 4$ matrix $J$\index{$J$} with integral entries $1\le J_{ij}\le r$.
\end{itemize}
\begin{notation}
By abuse of notation, we write $\square_{i1},\dotsc, \square_{i4}$\index{$q_{i1},\dotsc, q_{i4}$; $e_{i1},\dotsc, e_{i4}$, etc.} when addressing a quadruple of objects parameterized by the $i$-th row of matrix $J$; e.g, $q_{i1},\dotsc, q_{i4}$ or $e_{i1},\dotsc, e_{i4}$.
\end{notation}
Let $\Sigma\subset N_\RR$\index{$\Sigma$} be the fan generated by the rays $\rho_i:=\Span_{\RR_+}(n_i)$ for $1\le i\le r$. We denote by $X:=X_\Sigma$\index{$X$} the corresponding toric variety, and by $O_i$\index{$O_i$} the closure in $X$ of the $T_{L_i}$-orbit corresponding to $\zeta_i$. Then $O_i\cap X_{\rho_i}$ is given by $x^m=\zeta_i(m)$, $m\in L_i^0$, since $n_i\in L_i$, where $X_{\rho_i}$ is the open subscheme corresponding to the ray $\rho_i$. Finally, we set $\blambda^{\tr}:=\val(\blambda)\in \QQ^s$\index{$\lambda_i^{\tr}$} and $O^{\tr}_i:=\{m\mapsto\val(x^m(p))\,|\,p\in (T_N\cap O_i)(\oF)\}\subseteq N_\QQ$\index{$O_i^{\tr}$}. Notice that $O_i^{\tr}$ is the preimage of $\zeta_i^\tr:=\val(\zeta_i)\in N_\QQ/(L_i)_\QQ$\index{$\zeta_i^\tr$}.

\subsubsection{The objectives}
In this paper we compare between:

%\vspace{2mm}
%{\em On the algebraic side:} t
\noindent {\em The stable morphisms $f\:(C;\bq)\to X,$ where $(C;\bq)$ is a smooth irreducible rational curve with $r$ marked points such that}

%\vspace{2mm}
\begin{center}
\begin{tabular}{|c|c|}
  \hline
  {\bf Degree and tangency profile:} & $div(f^*x^m)=\sum (n_i,m)q_i$, \\
  \hline
  {\bf Toric constraint:} & $f(q_i)\in O_i$ for all $i\le r$, \\
  \hline
  {\bf Cross-ratio constraint:} & $\lambda(C;q_{i1},q_{i2},q_{i3},q_{i4})=\lambda_i$ for all $i\le s$; \\
  \hline
\end{tabular}
\end{center}

\noindent and {\em the stable rational $N_\QQ$-parameterized $\QQ$-tropical curves $h\:(\Gamma;\be)\to N_\RR$ with $r$ unbounded ends for which}

%\vspace{2mm}
\begin{center}
\begin{tabular}{|c|c|}
  \hline
  {\bf Degree and multiplicity profile:} & $h(u_{i})=n_i$ for all $i\le r$, \\
  \hline
  {\bf Affine constraint:} & $h(v_i)\in O_i^{\tr}$ for all $i\le r$, \\
  \hline
  {\bf Tropical cross-ratio constraint:} & $\lambda^{\tr}(\Gamma;e_{i1},e_{i2},e_{i3},e_{i4})=\lambda_i^{\tr}$ for all $i\le s$ \\
  \hline
\end{tabular}
\end{center}
%\vspace{2mm}

\begin{notation}\index{$\CW$}\index{$\CW^{\tr}$}
The sets of such morphisms and tropical curves are denoted by $\CW$ and $\CW^{\tr}$ respectively.
\end{notation}

The goal of the paper is to construct a natural map $\Tr\:\CW\to\CW^{\tr}$\index{$\Tr$} and to study its fibers. Under certain regularity conditions we prove that the fibers are non-empty, and count the number of points in each fiber.

\subsection{A toy example}\label{subsec:toy}
Let us consider the case of rational curves of class $(1,1)$ in $\PP^1\times\PP^1$ passing through two general points and having prescribed cross-ratio with respect to the four points in the complement of the big torus.

The setting is as follows: $M=N=\ZZ^2$, $r=6$, $n_1=\binom{1}{0}$, $n_2=-\binom{1}{0}$, $n_3=\binom{0}{1}$, $n_4=-\binom{0}{1}$, $n_5=n_6=0$, $L_i=N$ for $1\le i\le 4$ and $L_i=\{0\}$ for $i=5,6$, $\zeta_i=1$ for $1\le i\le 4$ and $\zeta_i\in (\overline{F}^\times)^2$ for $i=5,6$, $s=1$, $J=(1\;2\;3\;4)$, $\lambda_1=\lambda$. Thus, $X$ is the toric surface obtained by removing the four zero-dimensional orbits from $\PP^1\times\PP^1$, $O_i=X$ for $1\le i\le 4$ and $O_i=\zeta_i$ for $i=5,6$. We shall also assume that $\zeta_5,\zeta_6$, and $\lambda$ are general in an appropriate sense.

The space $\CW$ thus consists of the stable maps $f\:(C;q_1,\dotsc,q_6)\to \PP^1\times\PP^1$ for which $q_1=f^*(\{0\}\times\PP^1)$, $q_2=f^*(\{\infty\}\times\PP^1)$, $q_3=f^*(\PP^1\times\{0\})$, $q_4=f^*(\PP^1\times\{\infty\})$, $\lambda(C;q_1,\dotsc,q_4)=\lambda$, and $f(q_i)=O_i\in (\overline{F}^\times)^2$ for $i=5,6$.

It is easy to describe $\CW$ explicitly. Indeed, pick the coordinate $t$ on $C$ such that $t(q_1)=\lambda, t(q_2)=1, t(q_3)=0, t(q_4)=\infty$. Since $f$ is of class $(1,1)$ and $q_1,\dotsc,q_4$ are the pullbacks of the boundary divisor, $f$ is given by an equation of the form $t\mapsto \left(c_1\frac{t-\lambda}{t-1}, c_2t\right)$, for some $c_1,c_2\in \overline{F}$. Thus, the affine constraint is given by $\left(c_1\frac{t(q_i)-\lambda}{t(q_i)-1}, c_2t(q_i)\right)=O_i$ for $i=5,6$. After eliminating $t(q_5)$ and $t(q_6)$ one obtains one linear and one quadratic equation in $c_1,c_2$. Thus, for a general choice of $\lambda,O_5,O_6$, the space $\CW$ consists of two $\overline{F}$-points, and hence there exist precisely two rational curves of class $(1,1)$ in $\PP^1\times\PP^1$ satisfying the constraints.

Let us now describe the set $\CW^{\tr}$: it consists of the parameterized rational tropical curves with six unbounded ends $e_1,\dotsc e_6$, such that the slopes of $e_1,\dotsc, e_4$ are $\binom{1}{0}, -\binom{1}{0}, \binom{0}{1}, -\binom{0}{1}$ respectively, the ends $e_i$ are contracted to the points $\val(O_i)\in \RR^2$ for $i=5,6$, and $\lambda^{\tr}(\Gamma;e_1,e_2,e_3,e_4)=\val(\lambda)$. If $\val(\lambda)>0$ and $\val(O_5),\val(O_6)\in\RR^2$ are tropically general, e.g. $\val(O_5)=(0,0),\val(O_6)=(11,3), \val(\lambda)=5$, then one can check that there exist precisely two parameterized rational tropical curves in the set $\CW^{\tr}$ as in Figure~\ref{fig:picTrCRsEx}. It follows from the Correspondence theorem (Theorem~\ref{thm:correspondence}), that the map $\Tr\:\CW\to\CW^{\tr}$ is bijective in this particular example.

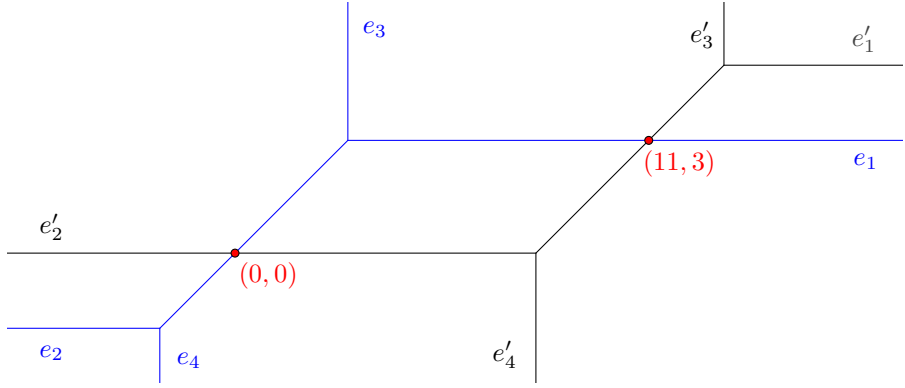
\begin{figure}
\definecolor{qqqqff}{rgb}{0.,0.,1.}
\definecolor{uuuuuu}{rgb}{0.26666666666666666,0.26666666666666666,0.26666666666666666}
\definecolor{ffqqqq}{rgb}{1.,0.,0.}
\definecolor{qqccqq}{rgb}{0.,0.8,0.}
\definecolor{zzttqq}{rgb}{0.6,0.2,0.}
\begin{tikzpicture}[line cap=round,line join=round,>=triangle 45,x=1.0cm,y=1.0cm]
\clip(-24.525,6.78) rectangle (-12.625,11.83);
%\fill[color=zzttqq,fill=zzttqq,fill opacity=0.1] (-27.,15.) -- (-27.,4.) -- (-11.,4.) -- (-11.,15.) -- cycle;
\draw [color=zzttqq] (-27.,15.)-- (-27.,4.);
\draw [color=zzttqq] (-27.,4.)-- (-11.,4.);
\draw [color=zzttqq] (-11.,4.)-- (-11.,15.);
\draw [color=zzttqq] (-11.,15.)-- (-27.,15.);
\draw [color=ffqqqq](-21.575,8.5) node[anchor=north west] {$(0,0)$};
\draw [color=ffqqqq](-16.2,10) node[anchor=north west] {$(11,3)$};
\draw [color=uuuuuu](-13.425,11.65) node[anchor=north west] {$e'_1$};
\draw (-24.225,9.15) node[anchor=north west] {$e'_2$};
\draw (-15.575,11.7) node[anchor=north west] {$e'_3$};
\draw (-18.2,7.45) node[anchor=north west] {$e'_4$};
\draw [color=qqqqff](-13.4,9.95) node[anchor=north west] {$e_1$};
\draw [color=qqqqff](-24.225,7.4) node[anchor=north west] {$e_2$};
\draw [color=qqqqff](-19.925,11.7) node[anchor=north west] {$e_3$};
\draw [color=qqqqff](-22.4,7.3) node[anchor=north west] {$e_4$};
\draw [color=qqqqff] (-26.425,7.5)-- (-22.5,7.5);
\draw [color=qqqqff] (-22.5,7.5)-- (-22.5,3.955);
\draw [color=qqqqff] (-22.5,7.5)-- (-20.,10.);
\draw [color=qqqqff] (-20,10)-- (-11,10);
\draw [color=qqqqff] (-20,10)-- (-20,15);
\draw (-27,8.5)-- (-17.5,8.5);
\draw (-17.5,8.5)-- (-17.5,4);
\draw (-17.5,8.5)-- (-15,11);
\draw (-15,11)-- (-11,11);
\draw (-15,11)-- (-15,15);
\begin{scriptsize}
\draw [fill=ffqqqq] (-21.5,8.5) circle (1.5pt);
\draw [fill=ffqqqq] (-16,10) circle (1.5pt);
\end{scriptsize}
\end{tikzpicture}
\caption{The toy example: the two curves in $\CW^{\tr}$.}
\label{fig:picTrCRsEx}
\end{figure}

\section{Tropicalization}\label{sec:trop}
We employ the canonical tropicalization of \cite{Tyo12} to construct the tropicalization map $\Tr\:\CW\to\CW^{\tr}$: Let $f\:(C;\bq)\to X$ be an element of $\CW$. Fix a complete discretely valued field of definition $K$ of $(C;\bq)$, and a uniformizer $\pi$. Since $(C;\bq)$ is rational, its stable model is defined over the ring of integers $K^o\subset K$. We define the underlying graph of the tropicalization $\Gamma$ to be the dual graph of the stable reduction of $(C;\bq)$. The vertices corresponding to the components of $C$ are called {\em finite}, and those that correspond to the marked points are {\em infinite}. The edges corresponding to nodes of the reduction are called {\em bounded}, and the edges corresponding to marked points are called {\em unbounded} (or {\em ends}). The length of an end is set to be $\infty$, and the length of the edge corresponding to a node $p$ of the reduction is defined to be $\val(\eta)$ if the total space of the stable model is given by $xy=\eta\in K^o$ in a (\'etale) neighborhood of $p$ (Figure~\ref{fig:dualgr}). The function $h$ is defined as follows: If a finite vertex $w$ corresponds to a component $C_w$ of the reduction then $h(w)(m):=\val(\pi)\cdot \ord_{C_w}(f^*(x^m))=\fe_K^{-1}\ord_{C_w}(f^*(x^m))$. If $u_i$ corresponds to the marked point $q_i$ then $h(u_i)(m):=\ord_{q_i}(f^*(x^m))$. The parametrization $h$ maps any finite vertex $w$ to $h(w)$, all bounded edges to the straight intervals joining the images of the attached vertices, and the ends $e_i$ to the rays $h(v_i)+\QQ_+h(u_i)$.

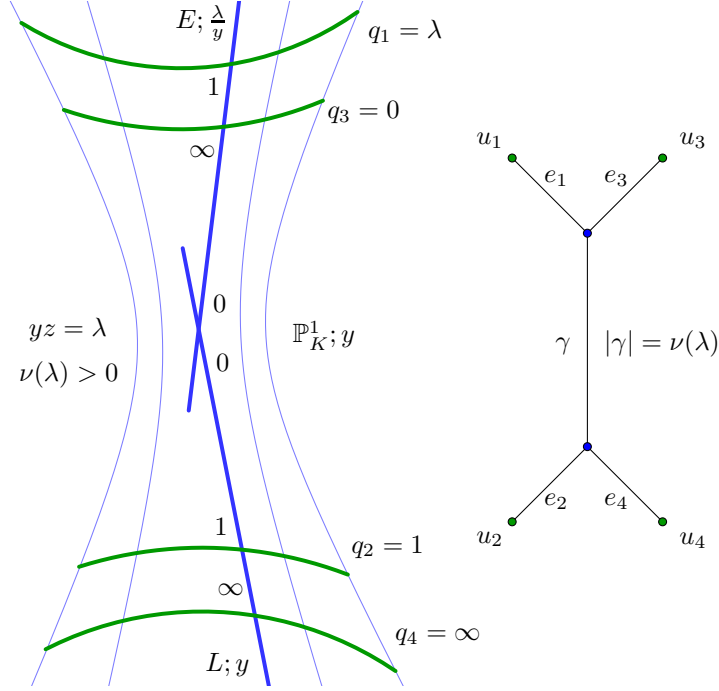
\begin{figure}
\definecolor{qqqqff}{rgb}{0.,0.,1.}
\definecolor{qqzzqq}{rgb}{0.,0.6,0.}
\definecolor{ttttff}{rgb}{0.2,0.2,1.}
\definecolor{xdxdff}{rgb}{0.49019607843137253,0.49019607843137253,1.}
\begin{tikzpicture}[line cap=round,line join=round,>=triangle 45,x=1.0cm,y=1.0cm]
\clip(0.02,-4.38) rectangle (9.8,4.72);
\draw [samples=50,domain=-0.99:0.99,rotate around={2.0771747651626784:(2.87,0.31)},xshift=2.87cm,yshift=0.31cm,color=xdxdff] plot ({0.853570104668427*(1+(\x)^2)/(1-(\x)^2)},{1.732402400257037*2*(\x)/(1-(\x)^2)});
\draw [samples=50,domain=-0.99:0.99,rotate around={2.0771747651626784:(2.87,0.31)},xshift=2.87cm,yshift=0.31cm,color=xdxdff] plot ({0.853570104668427*(-1-(\x)^2)/(1-(\x)^2)},{1.732402400257037*(-2)*(\x)/(1-(\x)^2)});
\draw [samples=50,domain=-0.99:0.99,rotate around={2.0771747651628667:(2.87,0.31)},xshift=2.87cm,yshift=0.31cm,color=xdxdff] plot ({0.5208046044953478*(1+(\x)^2)/(1-(\x)^2)},{1.859721098427524*2*(\x)/(1-(\x)^2)});
\draw [samples=50,domain=-0.99:0.99,rotate around={2.0771747651628667:(2.87,0.31)},xshift=2.87cm,yshift=0.31cm,color=xdxdff] plot ({0.5208046044953478*(-1-(\x)^2)/(1-(\x)^2)},{1.859721098427524*(-2)*(\x)/(1-(\x)^2)});
\draw [line width=1.5pt,color=ttttff,domain=2.7:9.800000000000015] plot(\x,{(-9.9828--3.58*\x)/0.44});
\draw [line width=1.5pt,color=ttttff,domain=2.62:9.800000000000015] plot(\x,{(--10.3352-3.56*\x)/0.7});
\draw [shift={(2.6,7.88)},line width=1.5pt,color=qqzzqq]  plot[domain=4.385372750972432:5.11115795334557,variable=\t]({1.*4.857593351218934*cos(\t r)+0.*4.857593351218934*sin(\t r)},{0.*4.857593351218934*cos(\t r)+1.*4.857593351218934*sin(\t r)});
\draw [shift={(2.6,7.88)},line width=1.5pt,color=qqzzqq]  plot[domain=4.159361824213505:5.330571452889091,variable=\t]({1.*4.045518284909221*cos(\t r)+0.*4.045518284909221*sin(\t r)},{0.*4.045518284909221*cos(\t r)+1.*4.045518284909221*sin(\t r)});
\draw [shift={(2.88,-7.96)},line width=1.5pt,color=qqzzqq]  plot[domain=1.205303649385764:1.878360327205055,variable=\t]({1.*5.415594819204725*cos(\t r)+0.*5.415594819204725*sin(\t r)},{0.*5.415594819204725*cos(\t r)+1.*5.415594819204725*sin(\t r)});
\draw [shift={(2.88,-7.96)},line width=1.5pt,color=qqzzqq]  plot[domain=0.9730834663547887:2.0443025384927265,variable=\t]({1.*4.5666229953795865*cos(\t r)+0.*4.5666229953795865*sin(\t r)},{0.*4.5666229953795865*cos(\t r)+1.*4.5666229953795865*sin(\t r)});
\draw (8,1.64)-- (8,-1.2);
\draw (8,1.64)-- (7,2.64);
\draw (8,1.64)-- (9,2.64);
\draw (8,-1.2)-- (7,-2.2);
\draw (8,-1.2)-- (9,-2.2);
\draw (4.96,4.6) node[anchor=north west] {$q_1=\lambda$};
\draw (4.42,3.52) node[anchor=north west] {$q_3=0$};
\draw (4.74,-2.26) node[anchor=north west] {$q_2=1$};
\draw (5.32,-3.5) node[anchor=north west] {$q_4=\infty$};
\draw (3.96,0.6) node[anchor=north west] {${\mathbb P}^1_K; y$};
\draw (2.82,3.83) node[anchor=north west] {$1$};
\draw (2.9,0.94) node[anchor=north west] {$0$};
\draw (2.58,2.93) node[anchor=north west] {$\infty$};
\draw (2.94,0.16) node[anchor=north west] {$0$};
\draw (2.92,-2.02) node[anchor=north west] {$1$};
\draw (2.96,-2.9) node[anchor=north west] {$\infty$};
\draw (2.8,-3.82) node[anchor=north west] {$L;y$};
\draw (2.4,4.8) node[anchor=north west] {$E;\frac{\lambda}{y}$};
\draw (6.4,3.1) node[anchor=north west] {$u_1$};
\draw (6.4,-2.2) node[anchor=north west] {$u_2$};
\draw (9.1,3.1) node[anchor=north west] {$u_3$};
\draw (9.1,-2.22) node[anchor=north west] {$u_4$};
\draw (8.1,0.48) node[anchor=north west] {$|\gamma|=\nu(\lambda)$};
\draw (0.45,0.64) node[anchor=north west] {$yz=\lambda$};
\draw (0.32,0.06) node[anchor=north west] {$\nu(\lambda)>0$};
%\draw (6.6,1.7) node[anchor=north west] {$v_1=v_3$};
%\draw (6.6,-0.8) node[anchor=north west] {$v_2=v_4$};
\draw (7.45,0.38) node[anchor=north west] {$\gamma$};
\draw (7.3,2.6) node[anchor=north west] {$e_1$};
\draw (7.3,-1.7) node[anchor=north west] {$e_2$};
\draw (8.1,2.6) node[anchor=north west] {$e_3$};
\draw (8.1,-1.7) node[anchor=north west] {$e_4$};
\begin{scriptsize}
\draw [fill=qqqqff] (8,1.64) circle (1.5pt);
\draw [fill=qqqqff] (8,-1.2) circle (1.5pt);
\draw [fill=qqzzqq] (7,2.64) circle (1.5pt);
\draw [fill=qqzzqq] (9,2.64) circle (1.5pt);
\draw [fill=qqzzqq] (7,-2.2) circle (1.5pt);
\draw [fill=qqzzqq] (9,-2.2) circle (1.5pt);
\end{scriptsize}
\end{tikzpicture}
\caption{The stable model of a rational curve with four marked points for which $\nu(\lambda(C;q_1,\dotsc,q_4))>0$ and its dual graph.}
\label{fig:dualgr}
\end{figure}

The curve $h\:(\Gamma;\be)\to N_\RR$ is a rational $N_\QQ$-parameterized $\QQ$-tropical curve by \cite[Lemma~2.23]{Tyo12}. Plainly, the degree and the multiplicity profile constraints are satisfied by the construction. To see that the affine constraint is satisfied, recall that $O_i\cap X_{\rho_i}$ is given by the equations $x^m=\zeta_i(m)$ for all $m\in L_i^0$, and $(n_i,m)=0$ for all $m\in L_i^0$ since $n_i\in L_i$. Thus, $h(v_i)(m)=\ord_{C_{v_i}}(f^*(x^m))=\ord_{q_i}(f^*(x^m))=\zeta_i(m)$ for all $m\in L_i^0$. Hence $h(v_i)\in O_i^{\tr}$. Finally, the tropical cross-ratio constraint is satisfied by the following:

\begin{lem}
Let $(C; q_1,\dotsc,q_4)$ be a smooth rational curve with four marked points over the field $K$, and $(\Gamma; e_1,\dotsc,e_4)$ its tropicalization. Then $\lambda^{\tr}=\val(\lambda)$.
\end{lem}

\begin{proof}
Choose the coordinate $y\:C\xrightarrow{\sim}\PP^1$ such that $(q_1, q_2, q_3, q_4)=(\lambda, 1, 0, \infty),$
and consider the trivial model $\PP^1_{K^o}$ over the ring of integers $K^o\subset K$. Assume, first, that the reduction of $q_1$ is different from $0$, $1$, and $\infty$. Then $\val(\lambda)=\val(1-\lambda)=0$, and $\PP^1_{K^o}$ is the stable model of $(C; q_1,\dotsc,q_4)$. Hence $\lambda^{\tr}=0=\val(\lambda)$.

Case 1: $\val(\lambda)>0$. To construct the stable model, consider the blow up $B$ of $\PP^1_{K^o}$ with respect to the ideal generated by $\lambda$ and $y$. Then the reduction of $B$ consists of the strict transform $L$ of $\PP^1_k$ and of the exceptional divisor $E\simeq\PP^1_k$. The two components intersect at one point, $q_1$ and $q_3$ specialize to distinct points of $E$, $q_2$ and $q_4$ to distinct points of $L$, and none of them to $E\cap L$. Hence $B$ is the stable model of $(C; q_1,\dotsc,q_4)$, see Figure~\ref{fig:dualgr}. Furthermore, the stable model is given by $yz=\lambda$ near the node of the reduction, and hence the length of the unique bounded edge of $\Gamma$ is $\val(\lambda)$. Thus, $\lambda^{\tr}=\val(\lambda)$.

Case 2: $\val(1-\lambda)>0$. In this case the stable model is the blow up of $\PP^1_{K^o}$ with respect to the ideal $(1-\lambda, 1-y)$. Moreover, $q_1$ and $q_2$ specialize to distinct points of the exceptional divisor, and $q_3$ and $q_4$ to distinct points of the strict transform of $\PP^1_k$. The length of the unique bounded edge of $\Gamma$ is $\val(1-\lambda)$. Thus, $\lambda^{\tr}=0=\val(\lambda)$.

Case 3: $\val(\lambda)<0$. This time the stable model is the blow up of $\PP^1_{K^o}$ with respect to $(\lambda^{-1}, 1/y)$. Furthermore, $q_1$ and $q_4$ specialize to distinct points of the exceptional divisor, and $q_2$ and $q_3$ to distinct points of the strict transform of $\PP^1_k$. The length of the unique bounded edge of $\Gamma$ is $-\val(\lambda)$. Thus, $\lambda^{\tr}=\val(\lambda)$.
\end{proof}

\section{Realization}
In this section we show that the image of the tropicalization map contains all {\em regular} curves, and describe the fiber of the tropicalization map over such curves. Let us start by reminding the notion of regularity.

Let $h\:(\Gamma;\be)\to N_\RR$ be an element of $\CW^{\tr}$. For $\gamma\in E^b(\Gamma)$ (resp. $\gamma\in E^\infty(\Gamma)$) set $n_\gamma:=\frac{h(\fh(\gamma))-h(\ft(\gamma))}{|\gamma|}$ (resp. $n_\gamma:=h(u)$, where $u$ is the infinite vertex attached to $\gamma$)\index{$n_\gamma$}. Then $n_\gamma\in N$ by the definition of parameterized tropical curves, and its integral length is called the {\em multiplicity} of $\gamma$. Caution! In \cite{Tyo12}, $n_\gamma$ denotes the primitive integral vector in the same direction. Consider now the following complex

\begin{equation}\label{eq:defreg}\index{$\theta$}
\theta\:\bigoplus_{w\in V^f(\Gamma)}N\oplus\bigoplus_{\gamma\in E^b(\Gamma)}\ZZ\to \bigoplus_{\gamma\in E^b(\Gamma)}N\oplus\bigoplus_{i=1}^r(N/L_i)\oplus\bigoplus_{i=1}^s\ZZ
\end{equation}
given by
$$1_\gamma\mapsto n_\gamma+\sum_{i=1}^s\epsilon(\gamma,i)\;\;\;\hbox{and}\;\;\; a_w\mapsto \sum_{\gamma}\epsilon(\gamma,w)a_w+\sum_{i=1}^r\delta(w,v_i)(a_w+L_i),$$
where $\epsilon(\gamma,i)$ is given by \eqref{eq:epsilonei}, and
$$\index{$\epsilon(\gamma,w)$}\index{$\delta(w,v_i)$}
\begin{array}{cc}
  \epsilon(\gamma,w)=\left\{
                  \begin{array}{ll}
                    1, & w=\ft(\gamma), \\
                    -1, & w=\fh(\gamma), \\
                    0 & \hbox{otherwise;}
                  \end{array}
                \right. &
  \delta(w,v_i)=\left\{
                  \begin{array}{ll}
                    1, & w=v_i, \\
                    0 & \hbox{otherwise;}
                  \end{array}
                \right.
\end{array}
$$

For an abelian group $G$, let $\theta_G$\index{$\theta_G$} be the map in the complex \eqref{eq:defreg}$\otimes_\ZZ G$. We denote
$$\CE^1_G(\Gamma, h; \bO^{\tr}, \blambda^{\tr}):=\Ker(\theta_G)\;\; \hbox{and}\;\; \CE^2_G(\Gamma, h; \bO^{\tr}, \blambda^{\tr}):={\rm Coker}(\theta_G).$$\index{$\CE^i_G(\Gamma, h; \bO^{\tr}, \blambda^{\tr})$}
If $G=\ZZ$ we omit $G$ in the notation of $\CE^\bullet_G$.
\begin{defin}[{cf. \cite[Definitions 2.45, 2.55]{Tyo12}}]
We say that $(\Gamma,h;\bO^{\tr},\blambda^{\tr})$ is {\em $G$-regular} if $\CE^2_G(\Gamma, h; \bO^{\tr}, \blambda^{\tr})=0$, and {\em $G$-superabundant} otherwise.
\end{defin}

\begin{rem}\label{rem:propE}
(i) Plainly, $G$-regularity is independent of the orientation of $\Gamma$.

(ii) By the structure theorem of finitely generated abelian groups, $\QQ$-regularity is equivalent to $\CE^2(\Gamma,h;\bO^{\tr},\blambda^{\tr})$ being a torsion group, since the latter is finitely generated, and tensor product preserves cokernels. Similarly, $k$-regularity is equivalent to $\CE^2(\Gamma,h;\bO^{\tr},\blambda^{\tr})$ being a torsion group of order prime to the characteristic of $k$. Hence $(\Gamma,h;\bO^{\tr},\blambda^{\tr})$ is $k$-regular if and only if it is $\QQ$-regular and the order of $\CE^2(\Gamma,h;\bO^{\tr},\blambda^{\tr})$  is not divisible by the characteristic of $k$.

(iii) If $(\Gamma,h;\bO^{\tr},\blambda^{\tr})$ is $\QQ$-regular then $\CE^2_{k^\times}(\Gamma,h;\bO^{\tr},\blambda^{\tr})=0$ since $k$ is algebraically closed. Moreover, the morphism of algebraic tori over $\Spec(k)$ corresponding to \eqref{eq:defreg} is surjective, the set of $k$-points of its kernel is $\CE^1_{k^\times}(\Gamma,h;\bO^{\tr},\blambda^{\tr})$, and the kernel is reduced if and only if the order of $\CE^2(\Gamma,h;\bO^{\tr},\blambda^{\tr})$ is not divisible by the characteristic. To see the latter, choose bases in \eqref{eq:defreg} such that the map is given by a matrix of the form $(0|D)$, where $D={\rm diag}(d_i)$ is diagonal. Then $\det(D)\ne 0$ by $\QQ$-regularity, $\CE^1_{k^\times}(\Gamma,h;\bO^{\tr},\blambda^{\tr})\simeq (k^\times)^\fr\times \prod_{i}\Ker(k^\times\xrightarrow{\;\square^{d_i}} k^\times)$, and the kernel is scheme-theoretically isomorphic to $T_{\ZZ,k}^\fr\times \prod_{i}\mu_{d_i,k}$, where $\fr$ is the rank of $\CE^1(\Gamma,h;\bO^{\tr},\blambda^{\tr})$, and $\mu_{d_i,k}\lhd T_{\ZZ,k}$ are the groups of roots of unity of orders $d_i$.

(iv) If $(\Gamma,h;\bO^{\tr},\blambda^{\tr})$ is $\QQ$-regular and $\CE^1(\Gamma,h;\bO^{\tr},\blambda^{\tr})=0$ then, using the notation above, $\fr=0$ and the scheme-theoretic length of the kernel $\prod_{i}\mu_{d_i,k}$ of the morphism of algebraic tori over $\Spec(k)$ corresponding to \eqref{eq:defreg} is equal to $\prod_{i}d_i=|\oplus_{i}\ZZ/d_i\ZZ|=|\CE^2(\Gamma,h;\bO^{\tr},\blambda^{\tr})|$. Further, by (ii) above, $(\Gamma,h;\bO^{\tr},\blambda^{\tr})$ is $k$-regular if and only if the characteristic of $k$ does not divide $\prod_i d_i$, or, equivalently $|\CE^1_{k^\times}(\Gamma,h;\bO^{\tr},\blambda^{\tr})|=\prod_{i}d_i=|\CE^2(\Gamma,h;\bO^{\tr},\blambda^{\tr})|$.
\end{rem}

\begin{thm}[Realization]\label{thm:realization}
Let $h\:(\Gamma;\be)\to N_\RR$ be an element of $\CW^{\tr}$, and $K\subset\oF$ a complete discretely valued subfield of definition of $\bO$, $\blambda$, and $(\Gamma, h)$. Assume that $\lambda_i^{\tr}\ne 0$ for all $i$, and $(\Gamma,h;\bO^{\tr},\blambda^{\tr})$ is $\QQ$-regular. Then

(1) $h\:(\Gamma;\be)\to N_\RR$ belongs to the image of $\Tr\:\CW\to\CW^{\tr}$.

(2) If $(\Gamma,h;\bO^{\tr},\blambda^{\tr})$ is $k$-regular, $\Gamma$ is three-valent, and $\CE^1(\Gamma, h; \bO^{\tr}, \blambda^{\tr})=0$ then the fiber of the tropicalization map $\Tr$ over
$h\:(\Gamma;\be)\to N_\RR$ consists of exactly $|\CE^2(\Gamma, h; \bO^{\tr}, \blambda^{\tr})|$ morphisms $f\:(C;\bq)\to X$, and all morphisms in the fiber are defined over $K$.
\end{thm}
\begin{rem}
It is sufficient to prove the theorem under the assumption that $\lambda_i^{\tr}>0$ for all $i$. Indeed, the general case reduces to this by replacing $\lambda_i$ with $\lambda_i^{-1}$ if $\lambda_i^{\tr}<0$, and switching the entries $(3,4)$ in the corresponding rows of matrix $J$. Thus, from now on {\em we will always assume that $\lambda_i^{\tr}>0$ for all $i$.}
\end{rem}

\subsection{The space of rational curves with given tropicalization}

In this subsection we introduce explicit coordinates on the space of rational curves with $r$ marked points tropicalizing to a given stable rational tropical curve with marked ends $(\Gamma;\be)$.

Let $(C;\bq)$ be a smooth rational curve with marked points tropicalizing to $(\Gamma;\be)$. For each finite vertex $w\in V^f(\Gamma)$, denote by $y_w$\index{$y_w$} the coordinate on $C$ such that
\begin{equation}\label{eq:defeqcoor}
y_w(q_r)=\infty, y_w(q_a)=0, y_w(q_b)=1,
\end{equation}
where $a$ and $b$ are the minimal and the maximal indices in $I_w$ respectively. Notice that $|I_w|=|E^+_w|=\deg(w)-1\ge 2$ since the curve $(\Gamma;\be)$ is stable, and hence \eqref{eq:defeqcoor} makes sense.

Let $\gamma\in E^b(\Gamma)$ be a bounded edge. If $a$ and $b$ are as above then $a=\iota(\gamma)$ and the coordinate $\frac{y_{\ft(\gamma)}-y_{\ft(\gamma)}(q_a)}{y_{\ft(\gamma)}(q_b)-y_{\ft(\gamma)}(q_a)}$ satisfies \eqref{eq:defeqcoor} for the vertex $w=\fh(\gamma)$. Hence
$$y_{\ft(\gamma)}-y_{\ft(\gamma)}(q_a)=(y_{\ft(\gamma)}(q_b)-y_{\ft(\gamma)}(q_a))y_{\fh(\gamma)}.$$ Thus, $|\gamma|=\val(y_{\ft(\gamma)}(q_b)-y_{\ft(\gamma)}(q_a))$ by the very definition of $|\gamma|$, since the stable model of $C$ is given by $(y_{\ft(\gamma)}-y_{\ft(\gamma)}(q_a))\frac{1}{y_{\fh(\gamma)}}=y_{\ft(\gamma)}(q_b)-y_{\ft(\gamma)}(q_a)$ in a neighborhood of the node corresponding to $\gamma$.

Finally, we define $\balpha^C\in T_\ZZ^{E^b(\Gamma)}(R)$ and $\bbeta^C\in \AAA^{E(\Gamma)}(R)$ by setting \index{$\alpha^C_\gamma$}\index{$\beta^C_\gamma$} $$\alpha^C_\gamma:=\frac{y_{\ft(\gamma)}(q_b)-y_{\ft(\gamma)}(q_a)}{\pi^{|\gamma|\fe_K}}\;\; \hbox{and}\;\;\beta^C_\gamma:=y_{\ft(\gamma)}(q_{\iota(\gamma)});$$ and for each $\gamma\in E^b(\Gamma)$ a linear function\index{$\Psi_\gamma$} $$\Psi_\gamma(y):=\beta^C_\gamma+\pi^{|\gamma|\fe_K}\alpha^C_\gamma y.$$

Let us summarize the properties of $\balpha=\balpha^C,\bbeta=\bbeta^C,$ and $\by$:

\begin{prop}\label{prop:coord}
(1) If $\gamma,\gamma'\in E^+_w$ are distinct then $\beta_\gamma-\beta_{\gamma'}\in T_\ZZ(R)$. Furthermore, if $\gamma\in E^+_w$ is minimal (resp. maximal) then $\beta_\gamma=0$ (resp. $\beta_\gamma=1$). %In particular, if $\deg(v)=3$ then $\{\beta(e)\}_{e\in E^+_v}=\{0,1\}$.

(2) If $\gamma\in E^b(\Gamma)$ then $y_{\ft(\gamma)}=\Psi_\gamma(y_{\fh(\gamma)}).$ In particular, for any $i$, $$(y_{\ft(\gamma)}-y_{\ft(\gamma)}(q_i))=\pi^{|\gamma|\fe_K}\alpha_\gamma(y_{\fh(\gamma)}-y_{\fh(\gamma)}(q_i)).$$

(3) Let $w, w'\in V^f(\Gamma)$ be finite vertices, $w=w_1,w_2,\dotsc,w_{d+1}=w'$ the vertices along the geodesic path $[w,w']$, and $\gamma_1,\dotsc, \gamma_d$ the corresponding edges. Then
$$y_w=\Psi_{\gamma_1}^{\epsilon_1}\circ\dots\circ\Psi_{\gamma_d}^{\epsilon_d}(y_{w'}),$$
where $\epsilon_i=1$ if $\gamma_i$ is directed along the path $[w,w']$ and $\epsilon_i=-1$ otherwise\footnote{Notice that since $\Gamma$ is a tree, there exists $j$ such that $\epsilon_i=1$ if and only if $i\ge j$.}.
%$$y_u=\beta(\varepsilon_1)+\pi^{|\varepsilon_1|\fe_K}\alpha(\varepsilon_1)\left(\beta(\varepsilon_2)+\pi^{|\varepsilon_2|\fe_K}\alpha(\varepsilon_2)\left(\dots\left(\beta(\varepsilon_{j-1}k)+\pi^{|\varepsilon_{j-1}|\fe_K}\alpha(\varepsilon_{j-1})y_{u'}\right)\right)\right).$$
In particular, for any $i$, $$(y_w-y_w(q_i))=\pi^{\fe_K\sum_{i=1}^d\epsilon_i|\gamma_i|}\left(\prod_{i=1}^d\alpha_{\gamma_i}^{\epsilon_i}\right)(y_{w'}-y_{w'}(q_i)).$$

(4) Let $1\le i<r$, $w\in V^f(\Gamma)$, $w=w_1,w_2,\dotsc,w_{d+1}=v_i$ the vertices along the geodesic path $[w,v_i]$, and $\gamma_1,\dotsc, \gamma_d$ the corresponding edges. Then
$$y_w(q_i)=\Psi_{\gamma_1}^{\epsilon_1}\circ\dots\circ\Psi_{\gamma_d}^{\epsilon_d}(\beta_{e_i}),$$
where $\epsilon_i$ are as in (3). In particular, if $w=\fh(e_r)=v_r$ then
\begin{equation}\label{eq:formulaqi}
y_{v_r}(q_i)=\beta_{\gamma_1}+\pi^{|\gamma_1|\fe_K}\alpha_{\gamma_1}\left(\dots(\beta_{\gamma_d}+\pi^{|\gamma_d|\fe_K}\alpha_{\gamma_d}\beta_{e_i})\right).
\end{equation}

(5) $\val(y_w(q_i))\ge 0$ if and only if $u_i\succ w$, for all $w\in V^f(\Gamma)$ and $1\le i\le r$.

(6) $\val(y_{\ft(\gamma)}(q_i)-\beta_\gamma)=0$ for all $\gamma\in E^b(\Gamma)$ and $i\in I_{\ft(\gamma)}^\infty\setminus I_{\fh(\gamma)}^\infty$.
\end{prop}
\begin{proof}
(1) and (2) follow from the definition of $\balpha$ and $\bbeta$, (3) follows from (2) by induction, (4) follows from (3) and the definition of $\bbeta$, (5) follows from (4), and (6) follows from (5) and (1).
\end{proof}
Set $v:=v_r$, and let $a,b\in I_v$ be the minimal and the maximal indices. Then the collection $\{y_v(q_i)\}$, $1\le i<r$, $i\ne a,b$, is a system of coordinates on the moduli space of all rational curves with $r$ marked points. One can check inductively on the partially ordered sets of edges and vertices that $\balpha^C$ and $\bbeta^C$ can be expressed explicitly in terms of $\{y_v(q_i)\}$ for curves tropicalizing to $(\Gamma;\be)$. Vice versa, \eqref{eq:formulaqi} expresses $\{y_v(q_i)\}$ in terms of $\balpha^C$ and $\bbeta^C$.

Thus, $(\alpha^C_\gamma,\beta^C_{\gamma'})_{\gamma\in E^b(\Gamma)}^{\gamma'\in E^{es}(\Gamma)}$ is a system of coordinates on the space of rational curves with $r$ marked points tropicalizing to $(\Gamma;\be)$, and $C\mapsto (\balpha^C,\bbeta^C)$ is an open immersion of this space into $T_\ZZ(R)^{|E^b(\Gamma)|+ov(\Gamma)}$. Moreover, it is fairly easy to describe the image of the immersion. Indeed, by Proposition~\ref{prop:coord} (1), if $\gamma_1<\dots<\gamma_d$ are the edges in $E_w^+$, $\beta_{\gamma_1}:=0, \beta_{\gamma_d}:=1$ then $\beta_{\gamma_i}-\beta_{\gamma_j}\in T_\ZZ(R)$ for all $i\ne j$. If the latter is satisfied for all $w\in V^f(\Gamma)$ then one defines $C:=\PP^1$, $q_r:=\infty$, and $q_i$ to be given by the right-hand side of \eqref{eq:formulaqi} for each $1\le i<r$. It is now straight-forward to verify that $(C;\bq)$ tropicalizes to $(\Gamma;\be)$ and $(\balpha^C,\bbeta^C)=(\balpha,\bbeta)$. We leave the details to the reader.

\subsection{The space of morphisms %$f\:(C;\bq)\to X$
with given tropicalization}

Let us fix a parameterized stable rational tropical curve with marked ends $h\:(\Gamma;\be)\to N_\RR$. In this subsection we give an explicit description of the space of morphisms $f\:(C;\bq)\to X$ tropicalizing to it, for which $(C;\bq)$ is a smooth connected rational curve with marked points.

Fix $(C;\bq)$ as above, and let $\balpha=\balpha^C$, $\bbeta=\bbeta^C$ be its coordinates. Let $\gamma\in E^b(\Gamma)$. Define $\phi_\gamma\in T_N(R)$\index{$\phi_\gamma$} by setting
$$
\phi_\gamma(m):=\frac{\prod_{i\in I_{\ft(\gamma)}^\infty\setminus I_{\fh(\gamma)}^\infty}\left(y_{\ft(\gamma)}(q_i)-\beta_\gamma\right)^{(n_i,m)}}{\prod_{i\notin I_{\ft(\gamma)}^\infty}\left(\frac{y_{\ft(\gamma)}(q_i)}{y_{\ft(\gamma)}(q_i)-\beta_\gamma}\right)^{(n_i,m)}}.
$$
%Notice, that $\val(y_{\ft(e)}(q_i)-\beta(e))=0$ for all $i\in I_{\ft(e)}^\infty\setminus I_{\fh(e)}^\infty$, $\val(y_{\ft(e)}(q_i))<0$ for all $i\notin I_{\ft(e)}^\infty$, and $\val(\beta(e))\ge 0$.
Notice that $\phi_\gamma$ is well defined since $\phi_\gamma(m)\in R^\times$ for all $m$ by Proposition~\ref{prop:coord} (5-6). Notice also that $\phi_\gamma$ is a function of $(\balpha,\bbeta)$ by Proposition~\ref{prop:coord} (4).

Let $f\:(C;\bq)\to X$ be a morphism tropicalizing to $h\:(\Gamma;\be)\to N_\RR$. For each $m\in M$ and $w\in V^f(\Gamma)$, let us express the rational function $f^*(x^m)$ in terms of the coordinate $y_w$. Since $x^m\in K(X)^\times$, $$f^*(x^m)\in \left(\oF[y_w, (y_w-y_w(q_{r+1}))^{-1},\dotsc, (y_w-y_w(q_{r-1}))^{-1}]\right)^\times.$$ Moreover,
\begin{equation}\label{eq:fmintermsofyw}\index{$\chi_w$}
f^*(x^m)=\pi^{\fe_Kh(w)(m)}\chi_w(m)\prod_{i\in I_w^\infty}(y_w-y_w(q_i))^{(n_i,m)}\prod_{i\notin I_w^\infty}\left(\frac{y_w}{y_w(q_i)}-1\right)^{(n_i,m)}
\end{equation}
since the boundary multiplicity profile of $(C,f)$ is $\{n_i\}$. Plainly $\chi_w\in T_N(\oF)$. Moreover, $\chi_w\in T_N(R)\subset T_N(\oF)$ by the definition of $h(w)$.

\begin{lem}%\label{lem:1}
For any $\gamma\in E^b(\Gamma)$ the following holds
\begin{equation}\label{eq:charcompitability}
\phi_\gamma\cdot\frac{\chi_{\ft(\gamma)}}{\chi_{\fh(\gamma)}}\cdot\alpha_\gamma^{n_\gamma}=1.
\end{equation}
\end{lem}
\begin{proof}
By \eqref{eq:fmintermsofyw} and Proposition~\ref{prop:coord} (2), the following holds for all $m\in M$:
$$
\frac{\pi^{\fe_Kh(\ft(\gamma))(m)}}{\pi^{\fe_Kh(\fh(\gamma))(m)}}\cdot\frac{\chi_{\ft(\gamma)}(m)}{\chi_{\fh(\gamma)}(m)}\cdot\frac{\prod_{i\notin I_{\fh(\gamma)}^\infty}y_{\fh(\gamma)}(q_i)^{(n_i,m)}}{\prod_{i\notin I_{\ft(\gamma)}^\infty}y_{\ft(\gamma)}(q_i)^{(n_i,m)}}\cdot\left(\pi^{|\gamma|\fe_K}\alpha_\gamma\right)^{\sum_i(n_i,m)}=1.
$$
Since $\sum_{i=1}^rn_i=0$ and $h(\fh(\gamma))-h(\ft(\gamma))=|\gamma|n_\gamma$, it remains to show that
\begin{equation}\label{eq:phie}
\phi_\gamma(m)=\frac{(\pi^{|\gamma|\fe_K}\alpha_\gamma)^{-(n_\gamma,m)}\prod_{i\notin I_{\fh(\gamma)}^\infty}y_{\fh(\gamma)}(q_i)^{(n_i,m)}}{\prod_{i\notin I_{\ft(\gamma)}^\infty}y_{\ft(\gamma)}(q_i)^{(n_i,m)}}.
\end{equation}
After summing up the balancing conditions over the finite vertices $w\nsucc \ft(\gamma)$, one obtains the following
$$
\sum_{i\notin I_{\fh(\gamma)}^\infty} n_i=-n_\gamma.
$$
Thus, the numerator of \eqref{eq:phie} is equal to $\prod_{i\notin I_{\fh(\gamma)}^\infty}\left(y_{\ft(\gamma)}(q_i)-\beta_\gamma\right)^{(n_i,m)}$ by Proposition~\ref{prop:coord} (2), which implies the lemma.
\end{proof}

Vice versa, if we are given a collection $\chi_w\in T_N(R)$ for all $w\in V^f(\Gamma)$ such that \eqref{eq:charcompitability} holds for all $\gamma\in E^b(\Gamma)$ then the maps given by \eqref{eq:fmintermsofyw} are compatible, and hence such datum defines a morphism $f\:(C;\bq)\to X$ tropicalizing to $h\:(\Gamma;\be)\to N_\RR$. We conclude that the space of such morphisms is given by equations \eqref{eq:charcompitability} in the trivial $T_N(R)^{V^f(\Gamma)}$-bundle over the space of curves $(C;\bq)$ tropicalizing to $(\Gamma;\be)$.

\subsection{The equations of the constraints}

Recall that $O_j$ is given by $x^m=\zeta_j(m)$ for all $m\in L_j^0$. Thus, the morphism $f\:(C;\bq)\to X$ satisfies the constraint $O_j$ if and only if $f^*(x^m)(q_j)=\zeta_j(m)$ for all $m\in L_j^0$.

Set $\zeta^\Gamma_j(m):=\pi^{-\fe_Kh(v_j)(m)}\zeta_j(m)$.\index{$\zeta^\Gamma_j$} Then $\zeta_j^\Gamma\in (T_N/T_{L_j})(R)$, and, by \eqref{eq:fmintermsofyw}, $f\:(C;\bq)\to X$ satisfies the constraint $O_j$ if and only if
$$\zeta^\Gamma_j(m)=\chi_{v_j}(m)\prod_{i\in I_{v_j}^\infty}(\beta_{e_j}-y_{v_j}(q_i))^{(n_i,m)}\prod_{i\notin I_{v_j}^\infty}\left(\frac{\beta_{e_j}}{y_{v_j}(q_i)}-1\right)^{(n_i,m)},$$
for all $m\in L_j^0$. Define $\varphi_j\in T_N(R)$\index{$\varphi_j$} by setting
$$
\varphi_j(m):=\prod_{i\in I_{v_j}^\infty}(\beta_{e_j}-y_{v_j}(q_i))^{(n_i,m)}\prod_{i\notin I_{v_j}^\infty}\left(\frac{\beta_{e_j}}{y_{v_j}(q_i)}-1\right)^{(n_i,m)}.
$$
Notice that $\varphi_j$ is well defined since $\varphi_j(m)\in R^\times$ for all $m$ by Proposition~\ref{prop:coord} (5-6). Notice also that $\varphi_j$ is a function of $(\balpha,\bbeta)$ by Proposition~\ref{prop:coord} (4). We conclude that $f\:(C;\bq)\to X$ satisfies the constraint $O_j$ if and only if
$$
\varphi_j\chi_{v_j}\equiv\zeta^\Gamma_j\mod T_{L_j}.
$$

Let us now reformulate the cross-ratio constraint: Pick an index $1\le i\le s$. Recall that we assumed that $\lambda_i^{\tr}>0$. Let $w_i,w'_i\in V^f(\Gamma)$ be such that $q_{i1},q_{i3}$ specialize to different points of the component $C_{w_i}$ and $q_{i2},q_{i4}$ specialize to different points of the component $C_{w'_i}$. Notice that an edge $\gamma$ belongs to the geodesic path $[w_i,w'_i]$ if and only if $\gamma$ separates $e_{i1},e_{i3}$ from $e_{i2},e_{i4}$. Let $w''_i$ be the minimal vertex along $[w_i,w'_i]$. Set $\lambda_i^{\Gamma}:=\pi^{-\fe_K\lambda_i^{\tr}}\lambda_i\in T_\ZZ(R)$\index{$\lambda_i^{\Gamma}$} and
$$
\psi_i:=\frac{(y_{w_i}(q_{i3})-y_{w_i}(q_{i1}))(y_{w'_i}(q_{i4})-y_{w'_i}(q_{i2}))}{(y_{w''_i}(q_{i4})-y_{w''_i}(q_{i1}))(y_{w''_i}(q_{i3})-y_{w''_i}(q_{i2}))}\in T_\ZZ(R).
$$\index{$\psi_i$}
Then, by Proposition~\ref{prop:coord} (3), $\lambda(C;q_{i1},\dotsc,q_{ir})=\lambda_i$ if and only if
$$
\psi_i\prod_{\gamma\subset[w_i,w'_i]}\alpha_\gamma=\lambda_i^\Gamma.
$$
As usual, $\psi_i$ can be expressed explicitly in terms of $(\balpha,\bbeta)$ by Proposition~\ref{prop:coord} (4).

\subsection{Proof of Theorem~\ref{thm:realization}}
Let $B\subset T_{\ZZ,K^o}^{E(\Gamma)}$\index{$B$} be the $K^o$-subscheme of points $\bbeta$ satisfying assertion (1) of Proposition~\ref{prop:coord}. Plainly, $B$ is flat over $K^o$ and has pure relative dimension $|E^{es}(\Gamma)|$. In particular, if $\Gamma$ is trivalent then $B\simeq\Spec(K^o)$.

Consider the $K^o$-morphism\index{$\Theta$}
$$
\Theta\:T_{N,K^o}^{V^f(\Gamma)}\times T_{\ZZ,K^o}^{E^b(\Gamma)}\times B\to T_{N,K^o}^{E^b(\Gamma)}\times\prod_{i=1}^r(T_{N,K^o}/T_{L_i,K^o})\times\prod_{i=1}^sT_{\ZZ,K^o}
$$
that maps $\left(\boldsymbol\chi,\boldsymbol\alpha,\boldsymbol\beta\right)$ to
$$\left(\left(\phi_\gamma(\balpha,\bbeta)\frac{\chi_{\ft(\gamma)}}{\chi_{\fh(\gamma)}}\alpha_\gamma^{n_\gamma}\right),\left(\varphi_j(\balpha,\bbeta)\chi_{v_j}\right),\left(\psi_i(\balpha,\bbeta)\prod_{\gamma\subset[w_i,w'_i]}\alpha_\gamma\right)\right)$$
where $w_i,w'_i\in V^f(\Gamma)$ are such that $q_{i1},q_{i3}$ specialize to different points of the component $C_{w_i}$ and $q_{i2},q_{i4}$ specialize to different points of the component $C_{w'_i}$. Let $\SSS$\index{$\SSS$} be the fiber of $\Theta$ over $(\b1,\bzeta^\Gamma,\blambda^\Gamma)$. Then $\CW=\SSS(R)$, and the ideal of $\SSS$ is generated by $l:=\dim\left(T_{N,K^o}^{E^b(\Gamma)}\times\prod_{i=1}^r(T_{N,K^o}/T_{L_i,K^o})\times\prod_{i=1}^sT_{\ZZ,K^o}\right)-\dim(K^o)$ functions. Let $f_1,\dotsc,f_l$ be such generators.
%$\rank\left(\bigoplus_{e\in E^b(\Gamma)}N\oplus\bigoplus_{i=1}^r(N/L_i)\oplus\bigoplus_{i=1}^s\ZZ\right)$

Let $\Theta_k=\Theta\times_{\Spec(K^o)}\Spec(k)$ be the reduction of $\Theta$. It follows from the definition and Proposition~\ref{prop:coord} that the reductions of $\phi_\gamma(\balpha,\bbeta)$, $\varphi_j(\balpha,\bbeta)$, and $\psi_i(\balpha,\bbeta)$ are independent of the reduction of $\balpha$. Thus, $\Theta_k$ is nothing but the product of the map
$B_k\to T_{N,k}^{E^b(\Gamma)}\times\prod_{i=1}^r(T_{N,k}/T_{L_i,k})\times\prod_{i=1}^sT_{\ZZ,k}$ given by the reduction of $(\bphi(\balpha,\bbeta), \bvarphi(\balpha,\bbeta), \bpsi(\balpha,\bbeta))$, and the homomorphism of $k$-algebraic groups $$\theta_{k^\times}\:T_{N,k}^{V^f(\Gamma)}\times T_{\ZZ,k}^{E^b(\Gamma)}\to T_{N,k}^{E^b(\Gamma)}\times\prod_{i=1}^r(T_{N,k}/T_{L_i,k})\times\prod_{i=1}^sT_{\ZZ,k}$$ corresponding to \eqref{eq:defreg}. By Remark~\ref{rem:propE} (iii), the later morphism is surjective, and the set of $k$-points of its kernel is $\CE^1_{k^\times}(\Gamma, h; \bO^{\tr}, \blambda^{\tr})$. Set $\GG:=\Ker(\theta_{k^\times})$\index{$\GG$}. Then the fibers of $\Theta_k$ are $\GG$-torsors over $B_k$. In particular, the fibers are irreducible of pure dimension $|E^{es}(\Gamma)|+\rank(\CE^1(\Gamma, h; \bO^{\tr}, \blambda^{\tr}))$. Thus, by \cite[Theorem~23.1]{MatCRT}, $\Theta$ is flat at any $k$-point, and hence so is the base change $\SSS\to \Spec(K^o)$.

(1) Since $\CW=\SSS(R)$, it is sufficient to construct a quasi-section of $\SSS\to \Spec(K^o)$, which exists by Mumford's existence theorem \cite[Proposition~14.5.10]{egaIV-III}. In our case, the construction is easy, and we include it for the convenience of the reader.

Let $p\in\SSS_k$ be a (general) closed point, and $g_1,\dotsc,g_d\in\fm_p\subset \CO_{\SSS,p}$ be such that their classes form a basis of the cotangent space at $p$ to the underlying reduced subscheme of $\SSS_k$. Then $d=\dim(\SSS_k)=|E^{es}(\Gamma)|+\rank(\CE^1(\Gamma, h; \bO^{\tr}, \blambda^{\tr}))$. Consider the reduced local subscheme $Z$ in the local scheme $\SSS_p$ defined by the functions $g_1,\dotsc,g_d$. Then by Hauptidealsatz (e.g., \cite[Theorem~13.5]{MatCRT}) the dimension of any component of $Z$ is at least $\dim\left(T_{N,K^o}^{V^f(\Gamma)}\times T_{\ZZ,K^o}^{E^b(\Gamma)}\times B\right)-l-d=1$. On the other hand, $\dim(Z_k)=0$. Thus, by Hauptidealsatz, $Z$ is equi-dimensional of dimension one, and $\pi$ is not a zero-divisor in $\CO(Z)$.

Let $Z'\to Z$ be the normalization of a component of $Z$. Then $\CO(Z')$ is integrally closed in its field of fractions $K(Z')$ and contains $K^o$. However, by \cite[Corollary~3 p.379 and Corollary~2 p.425]{Bou}, the integral closure of $K^o$ in a finite extension of $K$ is a discrete valuation ring, and hence a maximal proper subring. Thus, $\CO(Z')$ is a discrete valuation ring. Choose an embedding $K(Z')\hookrightarrow \oF$. Then $\CO(Z')=R\cap K(Z')$, and hence $Z'\to Z\subset\SSS$ defines a point in $\SSS(R)=\CW$ as needed.

(2) Since $K^o$ is complete and $k$ is algebraically closed, $\Spec(K^o)$ admits no non-trivial local \'etale coverings by \cite[Corollary~2 p.425]{Bou}. Thus, it is sufficient to show that $\SSS$ is \'etale over $K^o$ and $|\SSS_k|=|\CE^2(\Gamma, h; \bO^{\tr}, \blambda^{\tr})|$. By the assumption, $\Gamma$ is three-valent, and hence $B=\Spec(K^o)$. Furthermore $|\SSS_k|=|\CE^1_{k^\times}(\Gamma, h; \bO^{\tr}, \blambda^{\tr})|=|\CE^2(\Gamma, h; \bO^{\tr}, \blambda^{\tr})|$ by Remark~\ref{rem:propE} (iii)-(iv), since $(\Gamma,h;\bO^{\tr},\blambda^{\tr})$ is $k$-regular. Finally, the relative tangent space of $\Theta$ at each point $p\in\SSS_k$ is $\CE^1_k(\Gamma, h; \bO^{\tr}, \blambda^{\tr})=0$. Thus, $\SSS$ is \'etale over $\Spec(K^o)$, and we are done.

\begin{rem}\label{rem:smallchar}
If we omit the $k$-regularity assumption in (2) then $|\CE^2(\Gamma, h; \bO^{\tr}, \blambda^{\tr})|$ can still be interpreted as the number of algebraic curves in the fiber of $\Tr$ but counted with multiplicities. To see this, one considers the open neighborhood $\SSS'$ of $\SSS_k\subset\SSS$ that contains no irreducible components concentrated over the generic point of $\Spec(K^o)$. Then $\SSS'\to\Spec(K^o)$ is flat, and one can show that it is finite. Hence $|\CE^2(\Gamma, h; \bO^{\tr}, \blambda^{\tr})|$, which is equal to the length of $\SSS_k=\SSS'_k$ by Remark~\ref{rem:propE} (iv), is equal to the length of $\SSS'\times_{\Spec(K^o)}\Spec(K)$. The latter is precisely the number of curves in the fiber of $\Tr$ counted with the following multiplicities: the multiplicity of $f\:(C;\be)\to X$ is the length of the scheme concentrated at the class of the curve and cut out by the constraints on the moduli space of stable maps.
\end{rem}

\section{Correspondence}
In this section we prove the correspondence theorem under the assumption that the characteristic of $k$ is big enough, and $\bO^{\tr}$ and $\blambda^{\tr}$ are tropically general, by which we mean the following: If a family of objects is parameterized by a cone in a $\QQ$-affine space then an element in this family is {\em tropically general} for certain problem if it does not belong to a finite union of proper affine subspaces determined by the problem.

\begin{thm}[Correspondence]\label{thm:correspondence}
Assume that the constraints $\bO$ and $\blambda$ are such that $\bO^{\tr}$ and $\blambda^{\tr}$ are tropically general, and
\begin{equation}\label{eq:dim}
s+\sum_{i=1}^r\rank(N/L_i)=r-1.
\end{equation}
If the characteristic of $k$ is big enough (or zero) then the map $\Tr\:\CW\to\CW^{\tr}$ is surjective and the size of the fiber over $h\:(\Gamma; \be)\to N_\RR$ is $|\CE^2(\Gamma, h; \bO^{\tr}, \blambda^{\tr})|$. Moreover, all curves in the fiber are defined over any field of definition of $(\Gamma,h)$.
\end{thm}
\begin{defin}
The {\em complex multiplicity} of a curve $h\:(\Gamma; \be)\to N_\RR$ satisfying the constraints $\bO^{\tr}, \blambda^{\tr}$ is defined to be $m_\CCC(\Gamma, h; \bO^{\tr}, \blambda^{\tr}):=|\CE^2(\Gamma, h; \bO^{\tr}, \blambda^{\tr})|$\index{$m_\CCC(\Gamma, h; \bO^{\tr}, \blambda^{\tr})$}.
\end{defin}
Since the result of our enumerative problem over an algebraically closed field depends only on the characteristic we obtain the following:
\begin{cor}\label{cor:comcount}
Under the assumption of the Correspondence theorem, the number of stable complex rational curves $f\:(C;\bq)\to X$ satisfying general complex constraints $\bO,\blambda$ is equal to
$$\sum_{(\Gamma,h)\in\CW^\tr}m_\CCC(\Gamma, h; \bO^{\tr}, \blambda^{\tr}).$$
\end{cor}

The Correspondence theorem follows immediately from the Realization theorem using the following:
\begin{lem}\label{lem:trcurgenconstr}
Let $h\:(\Gamma;\be)\to N_\RR$ be an element of $\CW^{\tr}$. If the assumptions of the Correspondence theorem hold then $\Gamma$ is three-valent, $(\Gamma,h;\bO^{\tr},\blambda^{\tr})$ is $k$-regular, and $\CE^1(\Gamma,h;\bO^{\tr},\blambda^{\tr})=0$.
\end{lem}

\begin{proof}
Consider the linear map\index{$\vartheta$}
\begin{equation}\label{eq:defGQ}
\vartheta\:\bigoplus_{w\in V^f(\Gamma)}N\oplus\bigoplus_{\gamma\in E^b(\Gamma)}\ZZ\to \bigoplus_{\gamma\in E^b(\Gamma)}N
\end{equation}
given by composing the map $\theta$ of \eqref{eq:defreg} with the projection on the first summand. Then the space of parameterized tropical curves having the same combinatorial type, slopes, and multiplicities as the curve $h\:(\Gamma;\be)\to N_\RR$ can be identified naturally with the cone in $\CE_\QQ^1(\Gamma,h):=\Ker(\vartheta_\QQ)$ consisting of those elements whose projection to $\bigoplus_{\gamma\in E^b(\Gamma)}\QQ$ have no non-positive coordinates.

Since $\Gamma$ has no cycles $\bigoplus_{w\in V^f(\Gamma)}N_\QQ\to \bigoplus_{\gamma\in E^b(\Gamma)}N_\QQ$ is surjective, and hence so is \eqref{eq:defGQ}. Thus,
\begin{equation}\label{eq1}
\dim\left(\CE_\QQ^1(\Gamma,h)\right)=2|V^f(\Gamma)|-|E^b(\Gamma)|.
\end{equation}
Recall that for any tropical curve one has
\begin{equation}\label{eq2}
3|V^f(\Gamma)|+|V^\infty(\Gamma)|+ov(\Gamma)=2|E^b(\Gamma)|+2|E^\infty(\Gamma)|,\
\end{equation}
where $ov(\Gamma)=\sum_{w\in V^f(\Gamma)}(val(w)-3)$. In our case $\Gamma$ is stable and rational, thus $ov(\Gamma)\ge 0$ and
\begin{equation}\label{eq3}
|V^f(\Gamma)|=|E^b(\Gamma)|+1.
\end{equation}
After subtracting \eqref{eq3} from\eqref{eq2} and substituting to \eqref{eq1} one concludes:
$$
\dim\left(\CE_\QQ^1(\Gamma,h)\right)=|E^\infty(\Gamma)|-ov(\Gamma)-1=r-1-ov(\Gamma).
$$

Consider the natural projection
\begin{equation}\label{eq:mapphiq}\index{$\varrho$}
\varrho=\varrho^{(\Gamma,h)}\:\CE^1(\Gamma,h)\to\bigoplus_{i=1}^r(N/L_i)\oplus\bigoplus_{i=1}^s\ZZ.
\end{equation}
By the assumption, the curve $h\:(\Gamma;\be)\to N_\RR$ corresponds to the $\varrho_\QQ$-preimage of a general point. Since, there are only finitely many combinatorial types of  parameterized stable rational tropical curves with given degree and boundary multiplicity profile, we may assume that our constraint does not belong to the union of the spans of all non-maximal dimensional images of $\varrho_\QQ^{(\Gamma',h')}$. We conclude that $\varrho_\QQ$ is surjective, and hence $r-1-ov(\Gamma)\ge s+\sum_{i=1}^r \rank(N/L_i)=r-1$ by \eqref{eq:dim}. Thus, $ov(\Gamma)=0$, i.e., $\Gamma$ is three-valent, and $\varrho_\QQ$ is an isomorphism. Hence $|\CE^2(\Gamma, h; \bO^{\tr}, \blambda^{\tr})|<\infty$ and $\CE^1(\Gamma, h; \bO^{\tr}, \blambda^{\tr})=\Ker(\varrho)=0$. It follows now that $(\Gamma,h;\bO^{\tr},\blambda^{\tr})$ is $k$-regular as soon as the characteristic of $k$ does not divide the order of $\CE^2(\Gamma, h; \bO^{\tr}, \blambda^{\tr})$.
\end{proof}

\begin{rem}
Corollary~\ref{cor:comcount} reduces our complex enumerative problem to a tropical enumerative problem of finding an explicit description for $\CW^\tr$. A naive approach to the tropical problem would be the following:

Start by preparing a complete list of three-valent trees with $r$ ordered ends. Since the graph is a tree, either all its edges are ends or there exists a vertex whose star contains exactly one non-end. Thus, balancing condition determines the slope and the multiplicity of the non-end since the slopes and the multiplicities of the ends are given ($n_1,\dotsc, n_r$). Proceeding by induction, one obtains a combinatorial type of a parameterized stable rational tropical {\em pseudo-curve} $h\:(\Gamma; \be)\to N_\RR$, where by pseudo-curve we mean that the lengths of the bounded edges are allowed to be arbitrary reals. For a given combinatorial type, one considers the complex \eqref{eq:defreg}$\otimes_\ZZ\QQ$. Then the $E^b(\Gamma)$-components of $\theta^{-1}_\QQ(\bold0,\bzeta^\tr,\blambda^\tr)$ provide all possible ways to equip $\Gamma$ with edge lengths. Finally, thanks to Lemma~\ref{lem:trcurgenconstr}, if $\theta_\QQ$ is an isomorphism and all $E^b(\Gamma)$-components of $\theta^{-1}_\QQ(\bold0,\bzeta^\tr,\blambda^\tr)$ are strictly positive then the resulting curve $h\:(\Gamma; \be)\to N_\RR$ belongs to $\CW^\tr$, and all curves in $\CW^\tr$ are obtained this way.

The number of three-valent trees with $r$ ordered ends is $(2r-1)!!$ suggesting that the naive approach is not very efficient. Even in our toy example (\S\ref{subsec:toy}), the list consists of $105$ trees, only $2$ of which do actually correspond to the parameterized tropical curves in $\CW^\tr$ (Figure~\ref{fig:picTrCRsEx}). However, we believe that there exist more efficient combinatorial ways (similar to lattice path algorithm or floor diagrams) to exhibit the set $\CW^\tr$, which would provide an algorithm to compute $\sum_{(\Gamma,h)\in\CW^\tr}m_\CCC(\Gamma, h; \bO^{\tr}, \blambda^{\tr})$. Unfortunately, currently we do not know such an algorithm, and finding one is among our future projects. Let us conclude by mentioning that the computation of the multiplicities $m_\CCC(\Gamma, h; \bO^{\tr}, \blambda^{\tr})$ is simple since $m_\CCC(\Gamma, h; \bO^{\tr}, \blambda^{\tr})=|\CE^2(\Gamma, h; \bO^{\tr}, \blambda^{\tr})|=\det(\theta)$ by its very definition. In particular, in our toy example both multiplicities are $1$.
\end{rem}

\subsection{The real case}
In this subsection we assume that $k=\CCC$, $F=\CCC(\!(t)\!)$ is the field of Laurent series, and hence $\oF$ is the field of Puiseux series. Consider the natural involution $\varsigma$\index{$\varsigma$} acting by coefficient-wise complex conjugation, and let $\oF^\varsigma=\RR(\!(t)\!)$ be its fixed field. Assume that the toric and the cross-ratio constraints are defined over $\RR(\!(t)\!)$. The goal of this section is to describe the subset $\CW^\varsigma\subseteq\CW$ of $\varsigma$-invariant points under the assumptions of the Correspondence theorem. Since the action of $\varsigma$ preserves the fibers of the tropicalization map $\Tr\:\CW\to\CW^\tr$, we fix an element $h\:(\Gamma; \be)\to N_\RR$ of $\CW^\tr$, and analyze the action of $\varsigma$ on the corresponding fiber $\CW_{\Gamma,h}$.

Recall that by Lemma~\ref{lem:trcurgenconstr}, the assumptions of assertion (2) of Realization theorem are satisfied. Let $\SSS$ be as in the proof of the theorem. Then the points in $\CW_{\Gamma,h}=\SSS(R)$ are {\em uniquely} determined by their reduction in $\SSS_\CCC$. Notice that since the constraints are defined over $\oF^\varsigma$, $\SSS(R)$ admits a natural action of $\varsigma$ compatible with the complex conjugation on the reduction $\SSS_\CCC$. Hence $\varsigma$-invariant points of $\SSS(R)$ specialize to real points of $\SSS_\CCC$, and pairs of $\varsigma$-conjugate points to pairs of complex-conjugate points.

Let us now describe the set $\SSS_\CCC(\RR)$ of real points of $\SSS_\CCC$: Consider the exact sequence \eqref{eq:defreg}$\otimes_\ZZ\RR^\times$. Then $\SSS_\CCC(\RR)$ is the preimages of the reduction $\xi$\index{$\xi$} of $(\b1,\bzeta^\Gamma,\blambda^\Gamma)$. Thus, $\SSS_\CCC(\RR)\ne\emptyset$, if and only if the class of $\xi$ in $\CE^2_{\RR^\times}(\Gamma, h; \bO^{\tr}, \blambda^{\tr})$ is trivial, and in this case
$|\SSS_\CCC(\RR)|=|\CE^1_{\RR^\times}(\Gamma, h; \bO^{\tr}, \blambda^{\tr})|$.

Let us look closer at the groups $\CE^i_{\RR^\times}(\Gamma, h; \bO^{\tr}, \blambda^{\tr})$. Pick coordinates in \eqref{eq:defreg} such that the map $\theta$ is given by a diagonal matrix $D=diag(d_i)$. Then
\begin{equation}\label{eq:realnum}
\CE^i_{\RR^\times}(\Gamma, h; \bO^{\tr}, \blambda^{\tr})=\CE^i_{\RR^\times/\RR_+}(\Gamma, h; \bO^{\tr}, \blambda^{\tr})\approx \{\pm1\}^\varepsilon,
\end{equation}
for $i=1,2$, where $\varepsilon$ denotes the number of even $d_j$-s. Denote by $\sigma(\bO,\blambda)$\index{$\sigma(\bO,\blambda)$} the image of the reduction of $(\bzeta^\Gamma,\blambda^\Gamma)$ in $\left(\bigoplus_{i=1}^r(N/L_i)\oplus\bigoplus_{i=1}^s\ZZ\right)\otimes_\ZZ\left(\RR^\times/\RR_+\right)$. We shall call $\sigma(\bO,\blambda)$ {\em the sign} of the real constraint $(\bO,\blambda)$. Then the vanishing of the class of $\xi$ in $\CE^2_{\RR^\times}(\Gamma, h; \bO^{\tr}, \blambda^{\tr})$ depends only on $\Gamma,\bO^\tr,\blambda^\tr$ and the sign of the constraint, rather than on the constraint itself.

\begin{defin}
The {\em real multiplicity} $m_\RR(\Gamma, h; \bO^{\tr}, \blambda^{\tr}, \sigma(\bO,\blambda)\!)$\index{$m_\RR(\Gamma, h; \bO^{\tr}, \blambda^{\tr}, \sigma(\bO,\blambda))$} of the parameterized rational tropical curve $h\:(\Gamma; \be)\to N_\RR$ satisfying the constraints $\bO^{\tr}, \blambda^{\tr}$ with respect to the sign $\sigma(\bO,\blambda)$  is defined to be $0$ if the image of $(\b1,\sigma(\bO,\blambda))$ in $\CE^2_{\RR^\times/\RR_+}(\Gamma, h; \bO^{\tr}, \blambda^{\tr})$ is not trivial, and $2^\varepsilon$ otherwise, where $\varepsilon$ denotes the number of groups of even order in any decomposition $\CE^2(\Gamma, h; \bO^{\tr}, \blambda^{\tr})\approx\oplus\ZZ/d_i\ZZ$.
\end{defin}

\begin{cor}\label{cor:realcount}
Assume that $\bO^{\tr}, \blambda^{\tr}$ are tropically general and \eqref{eq:dim} holds. Then for any choice of the sign $\sigma(\bO,\blambda)$ there exist real constraints $\bO_\RR,\blambda_\RR$ with given sign such that the number of real stable maps $f\:(C;\bq)\to X$ satisfying the constraints $\bO_\RR,\blambda_\RR$ is equal to
$$\sum_{(\Gamma,h)\in\CW^\tr}m_\RR(\Gamma, h; \bO^{\tr}, \blambda^{\tr}, \sigma(\bO,\blambda)).$$
\end{cor}
\begin{proof}
Choose any constraint $\bO,\blambda$ with tropicalizations $\bO^\tr,\blambda^\tr$ and given sign such that all $\zeta_i,\lambda_i$ are {\em convergent fractional power series with real coefficients}. Then all curves in $\CW$ are defined over the subfield of convergent fractional power series, and there are only finitely many such curves. Thus, for real $t$ small enough the $\varsigma$-invariant curves specialize to real curves and the pairs of $\varsigma$-conjugate curves specialize to pairs of complex-conjugate curves satisfying the constraints $\bO_\RR,\blambda_\RR$. Since distinct curves specialize to distinct curves for $t$ small enough, all complex curves satisfying the constraints $\bO_\RR,\blambda_\RR$ are obtained this way by Corollary~\ref{cor:comcount}. In particular, the number of real stable maps is equal to the number of $\varsigma$-invariant curves, which is equal to $\sum_{(\Gamma,h)\in\CW^\tr}m_\RR(\Gamma, h; \bO^{\tr}, \blambda^{\tr}, \sigma(\bO,\blambda))$ by \eqref{eq:realnum}.
\end{proof}

\begin{rem}(i) By its very definition, the real multiplicity $m_\RR(\Gamma, h; \bO^{\tr}, \blambda^{\tr}, \sigma(\bO,\blambda))$ is bounded above by the complex multiplicity $m_\CCC(\Gamma, h; \bO^{\tr}, \blambda^{\tr})$, and the inequality is strict in many cases, e.g., if $m_\CCC(\Gamma, h; \bO^{\tr}, \blambda^{\tr})$ is not a power of $2$. In such cases the complex count differs from the real count for any ``tropical" position of real constraints, i.e., the position obtained by a small enough specialization of a tropically general $\RR(\!(t)\!)$-constraint as in the proof of the corollary.

(ii) If one chooses the sign to be totally positive, i.e., the unit element of the group $\left(\bigoplus_{i=1}^r(N/L_i)\oplus\bigoplus_{i=1}^s\ZZ\right)\otimes_\ZZ\left(\RR^\times/\RR_+\right)$, then the class of $\xi$ in $\CE^2_{\RR^\times}(\Gamma, h; \bO^{\tr}, \blambda^{\tr})$ is trivial for any $(\Gamma,h)\in\CW^\tr$. Thus, such choices of sign give rise to the maximal possible number of real stable maps $f\:(C;\bq)\to X$ satisfying real constraints in ``tropical" position.

(iii) As one expects, unlike the algebraically closed case, the answer to the real enumerative problem does depend on the position of the constraints. Indeed, it is easy to construct examples such that there exists $(\Gamma,h)\in\CW^\tr$ for which $|\CE^2(\Gamma, h; \bO^{\tr}, \blambda^{\tr})|$ is even. Since the curve $\Gamma$ is rational, the complex \eqref{eq:defreg} is quasi-isomorphic to a complex of the form
$$N\oplus\bigoplus_{\gamma\in E^b(\Gamma)}\ZZ\to \bigoplus_{i=1}^r(N/L_i)\oplus\bigoplus_{i=1}^s\ZZ.$$
Thus, $\left(\bigoplus_{i=1}^r(N/L_i)\oplus\bigoplus_{i=1}^s\ZZ\right)\otimes_\ZZ\RR^\times$ maps surjectively onto $\CE^2_{\RR^\times}(\Gamma, h; \bO^{\tr}, \blambda^{\tr})$, and we can choose the constraints $\bO',\blambda'$ with tropicalizations $\bO^\tr,\blambda^\tr$ such that the class of $\xi$ in $\CE^2_{\RR^\times}(\Gamma, h; \bO^{\tr}, \blambda^{\tr})$ is non-trivial. Hence $|\CW_{\Gamma,h}^\varsigma|=0$. Specializing $t$ to a small positive real number we obtain constraints $\bO'_\RR,\blambda'_\RR$ for which the number of real stable maps $f\:(C;\bq)\to X$ satisfying the constraints $\bO'_\RR,\blambda'_\RR$ is strictly smaller than in the case of the totally positive sign.
\end{rem}

\section{Afterword}
\subsection{More on parameterized tropical curves satisfying general constraints}
\begin{prop}\label{prop:trcurgenconstr}
Under the assumptions of Lemma~\ref{lem:trcurgenconstr}, let $w',w''\in V^f(\Gamma)$ be finite vertices, $\gamma'\in E(\Gamma)$ and edge, $[w',w'']$ the geodesic path joining $w'$ to $w''$, and $[w',\gamma']$ the geodesic path containing $\gamma'$ whose end points are $w'$ and one of the endpoints of $\gamma'$. Then
\begin{enumerate}
\item $h(w')=h(w'')$ if and only if $h$ contracts $[w',w'']$;
\item if $h(w')\in h(\gamma')$ then $h([w',\gamma'])$ is a straight interval;
\item the number of contracted edges attached to $w'$ is either one or three.
\end{enumerate}
\end{prop}

\begin{proof}
(1) The ``if'' part is clear. For the ``only if'' part, assume that $h(w')=h(w'')$, and let $\eta\in \CE_\QQ^1(\Gamma,h)\subset \bigoplus_{w\in V^f(\Gamma)}N_\QQ\oplus\bigoplus_{\gamma\in E^b(\Gamma)}\QQ$ be the class of the curve $h\:\Gamma\to N_\RR$. Then $\eta$ satisfies the equation
\begin{equation}\label{eq:constr}
\sum_{\gamma\subset[w',w'']} a_\gamma n_\gamma=0.
\end{equation}
But \eqref{eq:constr} depends only on the slopes and the multiplicities of $h\:\Gamma\to N_\RR$, and since \eqref{eq:mapphiq} is an isomorphism, $\eta\in \CE_\QQ^1(\Gamma,h)$ is general. Thus, \eqref{eq:constr} holds true identically on $\CE_\QQ^1(\Gamma,h)$. On the other hand, the projection $\CE_\QQ^1(\Gamma,h)\to \bigoplus_{\gamma\in E^b(\Gamma)}\QQ$ is surjective since so is $\bigoplus_{w\in V^f(\Gamma)}N_\QQ\to \bigoplus_{\gamma\in E^b(\Gamma)}N_\QQ$. Thus, \eqref{eq:constr} holds true identically on $\bigoplus_{\gamma\in E^b(\Gamma)}\QQ$, and hence $n_\gamma=0$ for all $\gamma\in [w',w'']$, i.e., $h$ contracts $[w',w'']$.

The proof of (2) is similar, but the analog of \eqref{eq:constr} is considered modulo the slope of $\gamma'$. We leave the details to the reader. Finally, (3) follows from the balancing condition since $\Gamma$ is three-valent by Lemma~\ref{lem:trcurgenconstr}.
\end{proof}
\subsection{The dual subdivision}

Under the assumptions of Lemma~\ref{lem:trcurgenconstr}, assume that $N\simeq\ZZ^2$, and let $\Delta$ be the Newton polygon of $h\:\Gamma\to\QQ^2$. One may hastily conclude that the dual subdivision of $\Delta$ consists of triangles and $2k$-gons with parallel opposite edges. However, this need not be the case since, for example, there may exist flattened three-valent vertices as the example on Figure~\ref{fig:pic1} shows. Nevertheless, it is fairly easy to show using Lemma~\ref{lem:trcurgenconstr} and Proposition~\ref{prop:trcurgenconstr} that the dual subdivision of $\Delta$ consists of $2k$-gons with $k$ pairs of parallel {\em equi-length} edges (Minkowski sums of $k$ intervals), triangles, trapezoids, pentagons (resp. hexagons) having two (resp. three) pairs of parallel edges.

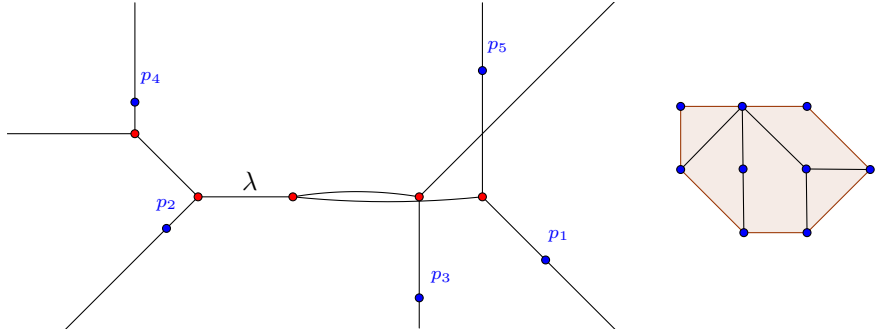
\begin{figure}
\definecolor{zzttqq}{rgb}{0.6,0.2,0.}
\definecolor{qqqqff}{rgb}{0.,0.,1.}
\definecolor{ffqqqq}{rgb}{1.,0.,0.}
\begin{tikzpicture}[line cap=round,line join=round,>=triangle 45,x=0.42cm,y=0.42cm]
\clip(-0.04,-5.1575) rectangle (30.155,5.1475);
\fill[color=zzttqq,fill=zzttqq,fill opacity=0.1] (21.275,1.865) -- (21.275,-0.135) -- (23.275,-2.135) -- (25.275,-2.135) -- (27.275,-0.135) -- (25.275,1.865) -- cycle;
\draw (4.,1.) -- (4.,5.1475);
\draw [domain=-0.040000000000004025:4.0] plot(\x,{(-3.-0.*\x)/-3.});
\draw (4.,1.)-- (6.,-1.);
\draw [domain=-0.040000000000004025:6.0] plot(\x,{(--7.-1.*\x)/-1.});
\draw (6.,-1.)-- (9.,-1.);
\draw (13.,-1.) -- (13.,-5.1575);
\draw [domain=13.0:30.15500000000002] plot(\x,{(-84.--6.*\x)/6.});
\draw (15.,-1.) -- (15.,5.1475);
\draw [domain=15.0:30.15500000000002] plot(\x,{(--28.-2.*\x)/2.});
\draw [shift={(11.,-13.)}] plot[domain=1.4056476493802699:1.7359450042095235,variable=\t]({1.*12.165525060596437*cos(\t r)+0.*12.165525060596437*sin(\t r)},{0.*12.165525060596437*cos(\t r)+1.*12.165525060596437*sin(\t r)});
\draw [shift={(12.,28.)}] plot[domain=4.6093073718619895:4.81547058890739,variable=\t]({1.*29.1547594742265*cos(\t r)+0.*29.1547594742265*sin(\t r)},{0.*29.1547594742265*cos(\t r)+1.*29.1547594742265*sin(\t r)});
\draw (7.1,0.1075) node[anchor=north west] {$\lambda$};
\draw [color=zzttqq] (21.275,1.865)-- (21.275,-0.135);
\draw [color=zzttqq] (21.275,-0.135)-- (23.275,-2.135);
\draw [color=zzttqq] (23.275,-2.135)-- (25.275,-2.135);
\draw [color=zzttqq] (25.275,-2.135)-- (27.275,-0.135);
\draw [color=zzttqq] (27.275,-0.135)-- (25.275,1.865);
\draw [color=zzttqq] (25.275,1.865)-- (21.275,1.865);
\draw (23.225,1.865)-- (21.275,-0.135);
\draw (23.225,1.865)-- (23.275,-2.135);
\draw (23.225,1.865)-- (25.25,-0.1175);
\draw (25.25,-0.1175)-- (27.275,-0.135);
\draw (25.25,-0.1175)-- (25.275,-2.135);
\begin{scriptsize}
\draw [fill=ffqqqq] (4.,1.) circle (1.5pt);
\draw [fill=qqqqff] (4.,2.) circle (1.5pt);
\draw[color=qqqqff] (4.515,2.7625) node {$p_4$};
\draw [fill=ffqqqq] (6.,-1.) circle (1.5pt);
\draw [fill=qqqqff] (5.,-2.) circle (1.5pt);
\draw[color=qqqqff] (5.005,-1.2425) node {$p_2$};
\draw [fill=ffqqqq] (9.,-1.) circle (1.5pt);
\draw [fill=ffqqqq] (13.,-1.) circle (1.5pt);
\draw [fill=ffqqqq] (15.,-1.) circle (1.5pt);
\draw [fill=qqqqff] (15.,3.) circle (1.5pt);
\draw[color=qqqqff] (15.495,3.7525) node {$p_5$};
\draw [fill=qqqqff] (17.,-3.) circle (1.5pt);
\draw[color=qqqqff] (17.42,-2.2325) node {$p_1$};
\draw [fill=qqqqff] (13.,-4.2) circle (1.5pt);
\draw[color=qqqqff] (13.7,-3.5425) node {$p_3$};
\draw [fill=qqqqff] (21.275,1.865) circle (1.5pt);
\draw [fill=qqqqff] (21.275,-0.135) circle (1.5pt);
\draw [fill=qqqqff] (23.275,-2.135) circle (1.5pt);
\draw [fill=qqqqff] (25.275,-2.135) circle (1.5pt);
\draw [fill=qqqqff] (27.275,-0.135) circle (1.5pt);
\draw [fill=qqqqff] (25.275,1.865) circle (1.5pt);
\draw [fill=qqqqff] (23.225,1.865) circle (1.5pt);
\draw [fill=qqqqff] (25.25,-0.1175) circle (1.5pt);
\draw [fill=qqqqff] (23.24978440868614,-0.11775269489142325) circle (1.5pt);
\end{scriptsize}
\end{tikzpicture}
\caption{A parameterized tropical curve in dimension 2 with a flattened three-valent vertex, and the corresponding dual subdivision of the Newton polygon for $r=5,d=7,s=1,p_{1i}=p_i$.}
\label{fig:pic1}
\end{figure}

\subsection{An algebraic approach to Realization theorem}\label{subsec:algproof}
The goal of this subsection is to give a sketch of an elementary algebraic proof of Realization theorem under the assumption that $(\Gamma, h; \bO^{\tr}, \blambda^{\tr})$ is $k$-regular. We shall only explain the proof of assertion (1). Assertion (2) can be proved along the same lines, and we leave it to the reader. We start in the same way we proved the theorem, and introduce the map $\Theta$ and the scheme $\SSS$ such that $\SSS(R)=\CW$.

The idea now is very simple: We think about $K^o$-points $\xi\in\SSS(K^o)$ as solutions of the system of equations $\Theta(\xi)=(\b1,\bzeta^\Gamma,\blambda^\Gamma)$, and we construct a sequence of approximate solutions, i.e., $\xi_d\in T_{N}^{V^f(\Gamma)}(K^o)\times T_{\ZZ}^{E^b(\Gamma)}(K^o)\times B(K^o)$\index{$\xi_d$} for $d\ge 0$, such that for all $d$ the following holds: $\xi_{d+1}\equiv \xi_d\mod (\pi^d)$ and
$$\Theta(\xi_d)\equiv (\b1,\bzeta^\Gamma,\blambda^\Gamma) \mod(\pi^d).$$
Thus, there exists $\xi:=\lim_{d\to\infty}\xi_d\in T_{N}^{V^f(\Gamma)}(K^o)\times T_{\ZZ}^{E^b(\Gamma)}(K^o)\times B(K^o)$ since $K^o$ is complete, and $\Theta(\xi)\equiv (\b1,\bzeta^\Gamma,\blambda^\Gamma) \mod(\pi^d)$ for all $d$. Hence $\xi\in\SSS(K^o)$ is the desired point.

Fix once and for all a point $\bbeta\in B$. We shall construct the sequence $\xi_d$, whose $B$-component is $\bbeta$ for all $d$. Denote by $\Theta_{\bbeta}$ the restriction of $\Theta$ to $(b=\bbeta)$-locus. By $k$-regularity, the reduction of $\Theta_{\bbeta}$ modulo $\pi$ is surjective, and hence there exists $\xi_0$ satisfying the requirements. Assume by induction that we have constructed $\xi_0,\dotsc, \xi_d$ as needed, and for $w\in V^f(\Gamma)$, $\gamma\in E^b(\Gamma)$, let $\chi_w\:N\to \GG_d$, $\alpha_\gamma\:\ZZ\to \GG_d$ be homomorphisms, where $\GG_d:=1+(\pi^{d+1})\lhd (K^0)^\times$\index{$\GG_d$}. We shall look for $\xi_{d+1}$ of the form $\xi_d\cdot(\bchi,\balpha,\b1)$. Plainly, any such $\xi_{d+1}$ satisfies: $\xi_{d+1}\equiv\xi_d\mod (\pi^d)$.
\begin{lem}\label{lem:fundlem}
For any $d\ge 0$, if $\balpha\equiv \balpha'\mod (\pi^d)$ and $\bbeta\equiv \bbeta'\mod (\pi^{d+1})$ then
$$\begin{array}{cc}
    \phi_\gamma(\balpha,\bbeta)\equiv \phi_\gamma(\balpha',\bbeta') & \mod (\pi^{d+1}),\\
    \varphi_j(\balpha,\bbeta)\equiv \varphi_j(\balpha',\bbeta') & \mod (\pi^{d+1}),\\
    \psi_i(\balpha,\bbeta)\equiv \psi_i(\balpha',\bbeta') & \mod (\pi^{d+1}).
  \end{array}
$$
\end{lem}

The proof of the lemma is straight-forward and is left to the reader. Let $\theta'$ be the map in \eqref{eq:defreg}$\otimes_\ZZ(K^o)^\times$. Then, by Lemma~\ref{lem:fundlem}, $$\Theta(\xi_d\cdot(\bchi,\balpha,\b1))\equiv\Theta(\xi_d)\cdot\theta'(\bchi,\balpha)\mod (\pi^{d+1}).$$
Thus, it is sufficient to find $(\bchi,\balpha)$ such that
\begin{equation}\label{eq:last}
\theta'(\bchi,\balpha)\equiv \left(\Theta(\xi_d)\right)^{-1}(\b1,\bzeta^\Gamma,\blambda^\Gamma)\mod (\pi^{d+1}).
\end{equation}
By $k$-regularity, the order of $\CE^2(\Gamma, h; \bO^{\tr}, \blambda^{\tr})\approx\oplus\ZZ/d_i\ZZ$ does not divide the characteristic of $k$, and hence $\CE^2_{\GG_d}(\Gamma, h; \bO^{\tr}, \blambda^{\tr})\approx(\oplus\ZZ/d_i\ZZ)\otimes_\ZZ (\GG_d)^\times=1$. Indeed, all we have to check is that the $d_i$-power maps $\GG_d\to \GG_d$, are surjective for all $i$. But $K^o$ is complete, and the denominators of the expansion of $(1+\pi^{d+1}y)^{1/{d_i}}$ do not vanish. Thus, the maps are surjective, and hence \eqref{eq:last} admits a solution as needed.

\printindex

\bibliographystyle{amsalpha}

\begin{thebibliography}{CDPR12}

\bibitem[Bou72]{Bou}
Nicolas Bourbaki, \emph{Elements of mathematics. {C}ommutative algebra},
  Hermann, Paris; Addison-Wesley Publishing Co., Reading, Mass., 1972,
  Translated from the French. \MR{0360549 (50 \#12997)}

\bibitem[CDPR12]{CDPR12}
Filip Cools, Jan Draisma, Sam Payne, and Elina Robeva, \emph{A tropical proof
  of the {B}rill-{N}oether theorem}, Adv. Math. \textbf{230} (2012), no.~2,
  759--776. \MR{2914965}

\bibitem[CH98a]{CH98-1}
Lucia Caporaso and Joe Harris, \emph{Counting plane curves of any genus},
  Invent. Math. \textbf{131} (1998), no.~2, 345--392. \MR{1608583 (99i:14064)}

\bibitem[CH98b]{CH98-3}
\bysame, \emph{Enumerating rational curves: the rational fibration method},
  Compositio Math. \textbf{113} (1998), no.~2, 209--236. \MR{1639187
  (99e:14061b)}

\bibitem[GM07]{GM07}
Andreas Gathmann and Hannah Markwig, \emph{The {C}aporaso-{H}arris formula and
  plane relative {G}romov-{W}itten invariants in tropical geometry}, Math. Ann.
  \textbf{338} (2007), no.~4, 845--868. \MR{2317753 (2008e:14075)}

\bibitem[GM08]{GM08}
\bysame, \emph{Kontsevich's formula and the {WDVV} equations in tropical
  geometry}, Adv. Math. \textbf{217} (2008), no.~2, 537--560. \MR{2370275
  (2010i:14113)}

\bibitem[Gro66]{egaIV-III}
Alexander Grothendieck, \emph{\'{E}l\'ements de g\'eom\'etrie alg\'ebrique.
  {IV}. \'{E}tude locale des sch\'emas et des morphismes de sch\'emas. {III}},
  Inst. Hautes \'Etudes Sci. Publ. Math. (1966), no.~28, 255. \MR{0217086 (36
  \#178)}

\bibitem[IKS13]{IKS13}
Ilia Itenberg, Viatcheslav Kharlamov, and Eugenii Shustin, \emph{Welschinger
  invariants of small non-toric {D}el {P}ezzo surfaces}, J. Eur. Math. Soc.
  (JEMS) \textbf{15} (2013), no.~2, 539--594. \MR{3017045}

\bibitem[KM94]{KM94}
Maxim Kontsevich and Yuri Manin, \emph{Gromov-{W}itten classes, quantum
  cohomology, and enumerative geometry}, Comm. Math. Phys. \textbf{164} (1994),
  no.~3, 525--562. \MR{1291244 (95i:14049)}

\bibitem[Kon95]{K95}
Maxim Kontsevich, \emph{Enumeration of rational curves via torus actions}, The
  moduli space of curves ({T}exel {I}sland, 1994), Progr. Math., vol. 129,
  Birkh\"auser Boston, Boston, MA, 1995, pp.~335--368. \MR{1363062 (97d:14077)}

\bibitem[Mat89]{MatCRT}
Hideyuki Matsumura, \emph{Commutative ring theory}, second ed., Cambridge
  Studies in Advanced Mathematics, vol.~8, Cambridge University Press,
  Cambridge, 1989, Translated from the Japanese by M. Reid. \MR{1011461
  (90i:13001)}

\bibitem[Mik05]{Mik05}
Grigory Mikhalkin, \emph{Enumerative tropical algebraic geometry in {$\RR^2$}},
  J. Amer. Math. Soc. \textbf{18} (2005), no.~2, 313--377. \MR{2137980
  (2006b:14097)}

\bibitem[Mik07]{Mik07}
\bysame, \emph{Moduli spaces of rational tropical curves}, Proceedings of
  {G}\"okova {G}eometry-{T}opology {C}onference 2006, G\"okova
  Geometry/Topology Conference (GGT), G\"okova, 2007, pp.~39--51. \MR{2404949
  (2009i:14014)}

\bibitem[NS06]{NS06}
Takeo Nishinou and Bernd Siebert, \emph{Toric degenerations of toric varieties
  and tropical curves}, Duke Math. J. \textbf{135} (2006), no.~1, 1--51.
  \MR{2259922 (2007h:14083)}

\bibitem[Pan97]{Pan97}
Rahul Pandharipande, \emph{Counting elliptic plane curves with fixed
  {$j$}-invariant}, Proc. Amer. Math. Soc. \textbf{125} (1997), no.~12,
  3471--3479. \MR{1423328 (98h:14068)}

\bibitem[Ran89]{Ran89}
Ziv Ran, \emph{Enumerative geometry of singular plane curves}, Invent. Math.
  \textbf{97} (1989), no.~3, 447--465. \MR{1005002 (90g:14039)}

\bibitem[Ran15]{Ran15}
Dhruv Ranganathan, \emph{Moduli of rational curves in toric varieties and
  non-archimedean geometry}, ArXiv e-prints (2015), 1--31,
  http://arxiv.org/abs/1506.03754.

\bibitem[Shu05]{Shu05}
Eugenii Shustin, \emph{A tropical approach to enumerative geometry}, Algebra i
  Analiz \textbf{17} (2005), no.~2, 170--214. \MR{2159589 (2006i:14058)}

\bibitem[Shu06]{Shu06}
\bysame, \emph{A tropical calculation of the {W}elschinger invariants of real
  toric del {P}ezzo surfaces}, J. Algebraic Geom. \textbf{15} (2006), no.~2,
  285--322. \MR{2199066 (2006k:14100)}

\bibitem[Tyo12]{Tyo12}
Ilya Tyomkin, \emph{Tropical geometry and correspondence theorems via toric
  stacks}, Math. Ann. \textbf{353} (2012), no.~3, 945--995. \MR{2923954}

\bibitem[Tyo13]{Tyo13}
\bysame, \emph{On {Z}ariski's theorem in positive characteristic}, J. Eur.
  Math. Soc. (JEMS) \textbf{15} (2013), no.~5, 1783--1803. \MR{3082243}

\bibitem[Vak00]{Vak00-2}
Ravi Vakil, \emph{The enumerative geometry of rational and elliptic curves in
  projective space}, J. Reine Angew. Math. \textbf{529} (2000), 101--153.
  \MR{1799935 (2001j:14072)}

\end{thebibliography}

\end{document}